\documentclass[12pt]{extarticle}
\usepackage{amsmath, amsthm, amssymb, hyperref, color}
\usepackage{graphicx}
\usepackage[all]{xypic}
\usepackage{verbatim}
\usepackage{tikz}
\oddsidemargin \evensidemargin
\usepackage{makecell}
\usepackage{array}
\newcolumntype{?}{!{\vrule width 1pt}}
\usepackage{color}
\newif\ifincludeprevious

\includeprevioustrue  

\newtheorem{theorem}{Theorem}
\numberwithin{theorem}{section}
\newtheorem{proposition}[theorem]{Proposition}
\newtheorem{lemma}[theorem]{Lemma}
\newtheorem{corollary}[theorem]{Corollary}

\newtheorem{remark}[theorem]{Remark}
\newtheorem{example}[theorem]{Example}
\newtheorem{conjecture}[theorem]{Conjecture}

\newcommand{\PP}{\mathbb{P}}
\newcommand{\RR}{\mathbb{R}}
\newcommand{\QQ}{\mathbb{Q}}
\newcommand{\CC}{\mathbb{C}}
\newcommand{\ZZ}{\mathbb{Z}}
\newcommand{\NN}{\mathbb{N}}
\newcommand {\cP} {\mathcal {P}}
\newcommand{\fM} {\mathfrak{M}}

\newcommand{\cA} {\mathcal{A}}

\date{}
 
\title{\textbf{Moment Varieties of Measures on Polytopes}}

\author{Kathl\'en Kohn, Boris Shapiro and Bernd Sturmfels}

\begin{document}

\maketitle

 \begin{abstract}
\noindent The uniform probability measure on a convex polytope
induces piecewise polynomial densities on its projections. For a fixed combinatorial type of simplicial polytopes,
the moments of these measures are rational functions in the vertex coordinates.
We study projective varieties that are parametrized by finite collections of
such rational functions. Our focus lies on determining the prime ideals
of these moment varieties. Special cases include  Hankel 
determinantal ideals for polytopal splines on line segments,
and the relations among multisymmetric functions given by
the cumulants of a simplex. In general, our moment varieties
are more complicated than in these two special cases. They offer 
challenges for both numerical and symbolic computing
in algebraic geometry.
  \end{abstract}

\section{Introduction}

Inverse moment problems for positive and real-valued measures have been an active area of research
since the 19th century when Stieltjes obtained
 first significant results in the one-dimensional case. One point of entry
 to this subject area is Schm\"udgen's textbook \cite{Schm}.

In applications one usually considers a restricted class of measures, e.g.~those with finite  or
low-dimensional support, Gaussian mixtures, unimodal measures, just to mention a few. The set of moments  is often restricted as well, e.g.~by degree or structure.
One classical situation occurs in logarithmic potential theory 
where one studies harmonic moments \cite{Nov}.
 Such restrictions reveal many interesting features,
 such as the non-uniqueness of a measure with given moments (cf.~\cite{BroSt}).
 Another feature that is important, but much less studied,~is the
   overdeterminacy of the moment problem. This arises from 
   relations among the moments.
   
We are interested in polynomial relations among moments 
of probability measures on $\RR^d$. Such relations
exist for many natural families of measures \cite{BuFrSh}. They define
the moment varieties of these families.
For finitely supported measures these are  the
 secant varieties of Veronese varieties \cite{LO}.
Moment varieties of Gaussians and their mixtures were 
characterized in \cite{AFS, ARS}.

In this paper we study moment varieties arising from realization spaces of
convex polytopes \cite{Ziegler}. 
If $P$ is a polytope in $\RR^d$, then  we
 write $\mu_P$ for the uniform probability distribution on $P$.
The {\em moments} of the distribution $\mu_P$ are the expected values of the monomials:
\begin{equation}
\label{eq:moment} m_{i_1 i_2 \cdots i_d}  \quad = \quad
\int_{\RR^d} w_1^{i_1} w_2^{i_2} \cdots w_d^{i_d}\, {\rm d} \mu_P 
\qquad {\rm for} \,\, i_1,i_2,\ldots, i_d \in \NN. 
\end{equation}
 The list of all moments  $\bigl( \,m_I \,: \,I \in \NN^d\, \bigr)$
 uniquely encodes  the polytope $P$ since any positive or real-valued measure with compact support
  is determined by its full list of moments. 
  
The inverse moment problem for polygons and polytopes is still largely unexplored.
It has appeared in logarithmic potential theory \cite {BroSt, PaS},
 and in connection with the mother body problem  \cite {Gu}. 
  Algorithms for  reconstructing $P$ from  its axial moments
  can be found in ~\cite{GLPR, GPSS}.
  A practical application for moments of planar polygons was suggested by
  Sharon and Mumford~\cite{ShMu} 
  as a tool to reconstruct arbitrary planar shapes from their fingerprints.
  
 To introduce our topic of investigation,  suppose that $P$ is a simplicial polytope 
 in $\RR^d$ with $n$ vertices, denoted $x_k = (x_{k 1},x_{k 2},\ldots,x_{k d})$
 for $k=1,2,\ldots,n$.
One can vary these vertices locally without changing the
 combinatorial type $\cP$ of the polytope $P$.
Following  \cite[Chapter 3]{Ziegler}, by the {\em combinatorial type}
 $\cP$ we mean the lattice of faces of $P$.
For  fixed $\cP$,
 each moment
$m_I$, for $I=(i_1,\dots , i_d)$, becomes a locally defined function of the $n \times d$-matrix $X = (x_{k l})$.
We shall see in Section~\ref{sec2} that this function is rational and therefore extends to a
dense set of matrices $X$.
Furthermore, it is homogeneous of degree $|I|$, i.e.~$m_I(t X) = t^{|I|} m_I(X)$.

\begin{figure}[h]
\includegraphics[width=0.3\textwidth]{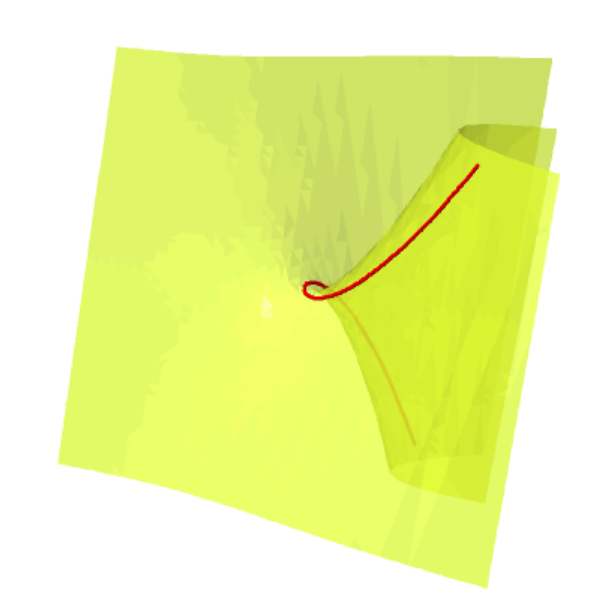} \qquad
\includegraphics[width=0.3\textwidth]{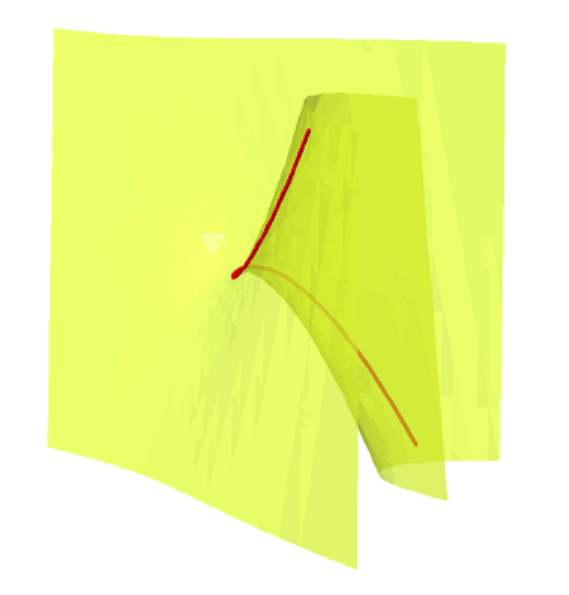} \quad
\includegraphics[width=0.3\textwidth]{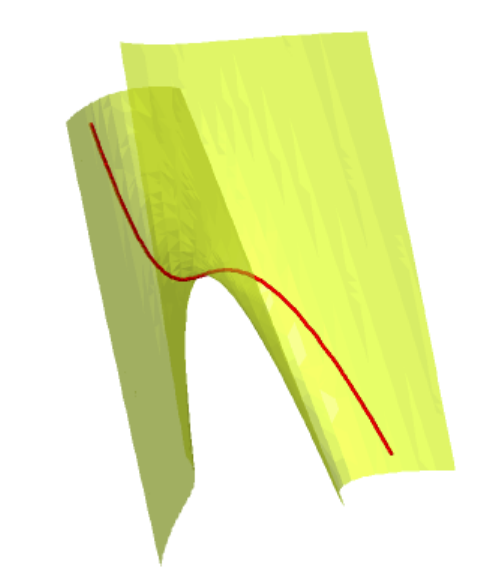}
\caption{The cubic surface (\ref{eq:cubicmm}) represents the first three moments 
(\ref{eq:1dimmoment}) of a line segment.
Segments of length zero correspond to points on the twisted cubic curve (shown in red).}
\label{fig:M13}
\end{figure}

\begin{example}[$d=1, n=2$]
\label{ex:d1n2} \rm 
The polytope is a segment $P = [a,b]$ on the real line  $\RR^1$.
Here $a = x_{11}$ and $b=x_{21}$.
The $i$th moment of the uniform distribution on $P$ is found by calculus:
\begin{equation}
\label{eq:1dimmoment}
 m_i \,\, = \,\, \frac{1}{b-a} \int_{a}^{b} \! w^i \,{\rm d}w \,\, = \,\,
\frac{1}{i{+}1}\frac{b^{i+1} - a^{i+1}}{b - a} \,\, = \,\,
\frac{1}{i{+}1} (a^i + a^{i-1} b + a^{i-2} b^2 + \cdots + b^i). 
\end{equation}
These expressions are the coefficients of the
{\em normalized moment generating function}
\begin{equation}
\label{eq:momentgenf2}
 \sum_{i=0}^\infty \, (i+1) \cdot m_i  \cdot t^i \quad = \quad
\frac{1}{(1 - a t)(1-b t)} . 
\end{equation}
The parametrization $(a,b) \mapsto  (m_0:m_1:\cdots:m_r)$
defines a surface in projective $r$-space $\PP^r$, for any $r \geq 3$.
The first such {\em moment surface}, shown in Figure~\ref{fig:M13},
is defined by the equation
\begin{equation}
\label{eq:cubicmm}
 2 m_1^3 - 3 m_0 m_1 m_2 + m_0^2 m_3\,\, = \,\, 0. 
\end{equation}
This cubic surface in $\PP^3$ is
 singular along the line $\{ m_0 = m_1 = 0\}$ in the plane at infinity.
 It also contains the twisted cubic curve
$\,\{m_0m_2 = m_1^2,\, m_0 m_3 = m_1m_2,  \, m_1 m_3 = m_2^2 \}$.
Points on that curve correspond to
degenerate line segments $[a,a]$ of length zero. 
\end{example}

The objects studied in this paper generalize Example \ref{ex:d1n2}.
We fix a combinatorial type $\mathcal{P}$ of simplicial $d$-polytopes
and a subset
$\mathcal{A} \subset \NN^d$ with $0 \not\in \mathcal{A}$.
Consider the semialgebraic set  of $n \times d$ matrices $X$ whose rows are the
vertices of a polytope of  type $\cP$.
This set is open in $\RR^{n \times d}$. Each moment  $m_I$ depends
rationally on $X$, so it extends to a unique rational function on $\CC^{n \times d}$.
The vector of moments $\bigr(\,m_I \,: \,I \in \mathcal{A} \cup \{0\} \,\bigr)$
 defines a rational map $\CC^{n \times d} \dashrightarrow \PP^{|\mathcal{A}|}$.
The {\em moment variety}  $\mathcal{M}_\mathcal{A}(\cP)$ is the
closure of the image of this map.
By construction, $\mathcal{M}_\mathcal{A}(\cP)$ is an irreducible
projective variety. Its dimension is $nd$, provided $\mathcal{A}$ is big enough.
Our aim is to compute these moment varieties as explicitly as~we~can.
Of particular interest is the variety given by all moments of order $\leq r$.
This is denoted
\begin{equation}
\label{eq:uptok}
\mathcal{M}_{[r]}(\cP) \,\,\subset \,\,\PP^{\binom{d+r}{d}-1} . 
\end{equation}

If $\cA$ lies in a coordinate subspace then 
we can reduce the dimensionality of our problem, 
but at the cost of passing to non-uniform measures on polytopes. 
 Suppose that $\cA\subset \NN^{d^\prime}$ for $ d^\prime<d$
 and let $\pi$ be the projection $\RR^d \rightarrow \RR^{d'}$.
 All moments  $m_I$ with $I \in \cA$ of the original polytope $P\subset \RR^d$ are moments
  of the induced distribution    $\pi_*(\mu_P)$   on  $P^\prime=\pi(P)$ in $\RR^{d'}$.
  Its density at $p \in P'$ is the $(d-d^\prime)$-dimensional volume of the inverse image $\pi^{-1}(p)\subset P$. In other words, $\pi_*(\mu_P)$      is the push-forward of $\mu_P$ under the projection $\pi$.  Densities of such measures are piecewise polynomial functions of degree $d-d^\prime$ and are called {\em polytopal splines}. They have been studied since the pioneering paper \cite{CuSch};  for more details consult \cite{DaMi,DCoPro}. 

This paper is organized as follows. 
In Section~\ref{sec2} we derive the parametric representation for our moment varieties.
It is encoded in a rational generating function 
(Theorem \ref{thm:nmgf_polytope}) whose numerator polynomial (\ref{eq:adjoint})
is Warren's adjoint from geometric modeling \cite{Warren1, Warren2}.
Section~\ref{sec3} concerns the univariate distribution obtained
by projecting $P$ onto a line segment $P'$. This corresponds to the above polytopal 
splines  with $d'=1$. Their moment varieties are
determinantal. 
Explicit Gr\"obner bases are furnished by the Hankel matrices in Theorem~\ref{th:Kathlen}.

In Section \ref{sec4} we examine the case  when $\mathcal{P}$ is the $d$-simplex.
We study moments and cumulants for 
uniform probability distributions on simplices, and we express
these as multisymmetric functions. This connects us to
an interesting, but notoriously difficult, subject in algebraic combinatorics.
Brill's equations  \cite{Guan} are used to characterize
the moment varieties of simplices. 
A Grassmannian makes a surprise appearance in
Proposition~\ref{prop:grass}.
The section concludes with tetrahedra: their moments of order $\leq 3$
form a $12$-dimensional variety in $\PP^{19}$.

Moment varieties of polytopes are invariant under affine transformations but
not under projective transformations. Section \ref{sec5} explores these
group actions and their invariant theory.
This is subtle because the affine group is not reductive, but
Theorem~\ref{thm:covariants} offers a remedy.

Section~\ref{sec6} presents a computer-aided case study for quadrilaterals.
We calculate their moment hypersurfaces in $\PP^9$.  This includes
the invariant hypersurface of degree $18$ in Theorem~\ref{thm:beaquad}.
We also examine concrete issues of identifiability and symmetry,
seen in the fiber of ten quadrilaterals in Figure \ref{fig:quads}, and in
relations among moments of tetrahedra in Proposition~\ref{prop:tetra}.
Some of our results are proved by certified numerical computations  as in \cite{Bertini}.

Section~\ref{sec7} offers a summary of this paper and an outlook for future directions.
Readers will find numerous open questions that arise from our
investigations in the earlier sections.

\medskip
\noindent

\bigskip \bigskip

\section{Generating Functions}
\label{sec2}

The moments of a polytope $P$ can be encoded in a rational generating function.
We begin by explaining this for the special case $n=d+1$, when 
 $P$ is a $d$-dimensional simplex $ \Delta_d$ in $\RR^d$.
 The vertices of the simplex $\Delta_d$ are denoted by $\,x_k = (x_{k1},\ldots,x_{kd})$
for $ k=1,2,\ldots, d+1$.

\begin{lemma}
\label{lem:simplex}
The moments $m_I$ of the uniform probability distribution on
the simplex $\Delta_d$  are obtained from the coefficients of
the {\em normalized moment generating function}
\begin{equation}
\label{eq:momentgenf} \!\!
\prod_{k=1}^{d+1} \frac{1}{1 - (x_{k1} t_1 {+} x_{k2} t_2 {+} \cdots {+} x_{kd} t_d)} \,\,\, = \,
\sum_{i_1,i_2,\ldots,i_d \in \NN} \!\!\! \frac{(i_1 {+} i_2 {+} \cdots {+} i_d +d)!}{i_1 ! \, i_2 !\, \cdots\, i_d! \,d !}
m_{i_1 i_2 \cdots i_d} t_1^{i_1} t_2^{i_2} \cdots t_d^{t_d}. 
\end{equation}
Each moment $m_I$ is a homogeneous polynomial of degree $| I |$
 in the $d^2+d$ unknowns $x_{kl}$.
This polynomial is multisymmetric:
it is invariant under permuting the vertices $x_1,x_2,\ldots,x_{d+1}$.
\end{lemma}

\begin{proof}
This can be found in several sources, e.g.~\cite[Theorem 10]{BBDKV} and \cite[Corollary 3]{GPSS}.
 \end{proof}

Observe that the normalized moment generating function (\ref{eq:momentgenf})
is different from the  standard {\em exponential moment generating function},
 commonly used in statistics and probability:
\begin{equation}
\label{eq:momentA}
\sum_{i_1,i_2,\ldots,i_d \in \NN}  \frac{m_{i_1 i_2 \cdots i_d}}
{i_1 ! \, i_2 !\, \cdots\, i_d! }  t_1^{i_1} t_2^{i_2} \cdots t_d^{t_d}. 
\end{equation}
This is the exponential version of the {\em ordinary generating function}
\begin{equation}
\label{eq:momentB}
\sum_{i_1,i_2,\ldots,i_d \in \NN}  m_{i_1 i_2 \cdots i_d}
 t_1^{i_1} t_2^{i_2} \cdots t_d^{t_d}. 
\end{equation}
The reason why we prefer (\ref{eq:momentgenf}) over these is that
 (\ref{eq:momentA}) and (\ref{eq:momentB}) are not rational functions.
This can be seen already for $d=1$ and $n=2$ as  in Example \ref{ex:d1n2}.
In that case, (\ref{eq:momentgenf})
is the rational function in (\ref{eq:momentgenf2}), whereas
the two other series (\ref{eq:momentA}) and (\ref{eq:momentB}) are 
the non-rational functions
\begin{equation}
\label{eq:nonrational}
\sum_{i=0}^\infty \frac{m_i}{i !} t^i \,=\,
\frac{{\rm exp}(bt) - {\rm exp}(at)}{(b-a) t} \qquad {\rm and} \qquad
\sum_{i=0}^\infty m_i t^i \,= \, 
\frac{{\rm log}(1-ta)- {\rm log}(1-tb)}{(b-a)t} .
\end{equation}

\smallskip

Let $P$ be a full-dimensional simplicial  polytope in $\RR^d$
with vertices $x_1,\ldots,x_n$, where $n \geq d+2$.
Fix any triangulation $\Sigma$ of $P$ that uses only these vertices.
We identify $\Sigma$ with the collection of subsets
$\sigma = \{\sigma_0,\ldots,\sigma_d\}$ that index the maximal simplices
${\rm conv}(x_{\sigma_0}, \ldots,x_{\sigma_d})$. The volume
of $P$ equals the sum of the volumes of these simplices. We write this as
$$ {\rm vol}(P) \,  = \, \sum_{\sigma \in \Sigma} {\rm vol}(\sigma). $$
If $\mu_\sigma$ denotes the uniform probability distribution on each simplex $\sigma$ then we have
$$ \mu_P \, = \, \frac{1}{{\rm vol}(P)} \sum_{\sigma \in \Sigma} {\rm vol}(\sigma) \mu_\sigma. $$
Since the moments depend linearly on the measure, 
Lemma \ref{lem:simplex} implies the following result:

\begin{theorem} \label{thm:nmgf_polytope}
The normalized moment generating function for the uniform probability distribution  $\mu_P$  
on the simplicial polytope $P$ is equal to
$$
\frac{1}{{\rm vol}(P)}
\sum_{\sigma \in \Sigma} \frac{{\rm vol}(\sigma)}{ 
\prod_{k \in \sigma} (1- x_{k 1} t_1 - x_{k 2} t_2 - \cdots - x_{k d} t_d)}
\,\,\,  = \,\,
\sum_{i_1,\ldots,i_d \in \NN} \!\!\! \! \frac{(i_1 {+}  \cdots {+} i_d {+}d)!}{i_1 ! \, \cdots\, i_d! \,d !}
m_{i_1  \cdots i_d} t_1^{i_1} \cdots t_d^{t_d}. 
$$
This expression is independent of the triangulation $\Sigma$. The coefficient $m_I = m_I(X)$ of \,$t^I$
is a rational function whose numerator is a homogeneous polynomial of degree $| I | + d$ in $X$ and whose
denominator equals ${\rm vol}(P)$, which is a homogeneous polynomial of degree $d$ in $X$.
\end{theorem}

To highlight the complexity of these moments, we examine the smallest non-simplex case.

\begin{example}[$d=2, n=4$] \rm
\label{ex:quadrMoments}
The polytope $P$ is a quadrilateral in the plane, with
cyclically labeled vertices $x_1,x_2,x_3,x_4$. The
moments of its uniform probability distribution $\mu_P$ are rational functions in eight unknowns 
$x_{kl}$.
The area of the quadrilateral is the quadratic form
$$ {\rm vol}(P) \,\,=\,\, \frac{1}{2}( x_{11} x_{22} {-} x_{12} x_{21} 
+ x_{21} x_{32} {-} x_{22} x_{31} + x_{31} x_{42} {-} x_{32} x_{41}
 +  x_{41} x_{12}  {-} x_{42}   x_{11} ). $$
The mean vector of the distribution $\mu_P$ is the centroid 
$(m_{10}, m_{01}) = \frac{1}{2 {\rm vol}(P)}(M_{10}, M_{01}) $, where
$$ \begin{small}
\begin{matrix} M_{10} & =   & \,\,
(x_{41}-x_{21}) (x_{41}+x_{11}+x_{21}) x_{12}   \,\,+\,\,
(x_{11}-x_{31}) (x_{11}+x_{21}+x_{31}) x_{22}  \,\,+ \\ & & \!\!
(x_{21}-x_{41}) (x_{21}+x_{31}+x_{41}) x_{32} \,\,+\,\,  
(x_{31}-x_{11}) (x_{31}+x_{41}+x_{11}) x_{42}, \,\smallskip \\
 M_{01}  & = &\,\,
(x_{22}-x_{42}) (x_{42}+x_{12}+x_{22})x_{11} \,\,+ \,\,
(x_{32}-x_{12}) (x_{12}+x_{22}+x_{32})x_{21} \,\,+  \\ && \!\!
(x_{42}-x_{22}) (x_{22}+x_{32}+x_{42})x_{31} \,\,+ \,\,
(x_{12}-x_{32}) (x_{32}+x_{42}+x_{12})x_{41}.\quad
\end{matrix}
\end{small}
$$
The {\em covariance matrix} of the distribution $\mu_P$ equals
$$ \begin{pmatrix} m_{20} & m_{11} \\ m_{11} & m_{02} \end{pmatrix} \quad = \quad
\frac{1}{24 \cdot {\rm vol}(P)} \cdot \begin{pmatrix} 2 M_{20} &  M_{11} \\ M_{11} & 2 M_{02} \end{pmatrix}, $$
where
 $$
 \begin{small}
 \begin{matrix} 
 M_{20} & = &
x_{12}(x_{41}-x_{21})(x_{11}^2+x_{11}x_{21}+x_{11}x_{41}+x_{21}^2+x_{21}x_{41}+x_{41}^2)\, + \\ & & 
x_{22}(x_{11}-x_{31})(x_{11}^2+x_{11}x_{21}+x_{11}x_{31}+x_{21}^2+x_{21}x_{31}+x_{31}^2) \,+ \\  & & 
x_{32}(x_{21}-x_{41})(x_{21}^2+x_{21}x_{31}+x_{21}x_{41}+x_{31}^2+x_{31}x_{41}+x_{41}^2)	\,+ \\ & & 
x_{42}(x_{31}-x_{11})(x_{11}^2+x_{11}x_{31}+x_{11}x_{41}+x_{31}^2+x_{31}x_{41}+x_{41}^2). \,\,\,\,
\end{matrix}
\end{small}
$$
The other diagonal entry $M_{02}$ is similar, and the off-diagonal entry equals
$$
\begin{small}
\begin{matrix}
M_{11} & = & 
 (x_{11}x_{22}-x_{12}x_{21}) (2x_{11}x_{12}+x_{11}x_{22}+x_{12}x_{21}+2x_{21}x_{22}) \, + \\ & & 
 (x_{21}x_{32}-x_{22}x_{31}) (2x_{21}x_{22}+x_{21}x_{32}+x_{22}x_{31}+2x_{31}x_{32})\, + \\ & & 
 (x_{31}x_{42}-x_{32}x_{41}) (2x_{31}x_{32}+x_{31}x_{42}+x_{32}x_{41}+2x_{41}x_{42})\,+ \\ & & 
 (x_{12}x_{41}-x_{11}x_{42}) (2x_{11}x_{12}+x_{11}x_{42}+x_{12}x_{41}+2x_{41}x_{42}).\,\,\,\,
 \end{matrix}
 \end{small}
$$
In Section \ref{sec6} we shall examine the relations satisfied by higher moments of quadrilaterals.
\end{example}

Let us return to Theorem  \ref{thm:nmgf_polytope} and take a closer look
at the rational function seen there. The normalized moment generating function 
can be written with a common denominator
\begin{equation}
\label{eq:momentgenf3}
\frac{ {\rm Ad}_P( t_1,t_2,\ldots,t_d)}{\prod_{k=1}^n (1- x_{k1} t_1 - x_{k2} t_2 - \cdots - x_{kd} t_d)}.
\end{equation}
The numerator is an inhomogeneous polynomial of degree at most $n-d-1$ in the variables $t_1,t_2,\ldots, t_d$. Its coefficients are rational
functions in the entries of the $n \times d$ matrix $X = (x_{kl})$:
\begin{equation}
\label{eq:adjoint} {\rm Ad}_P( t_1,t_2,\ldots,t_d) \,\,\,= \,\,\, 
\sum_{\sigma \in \Sigma} \frac{{\rm vol}(\sigma)}{{\rm vol}(P)}
\prod_{k \not\in \sigma}\bigl(1- x_{k1} t_1 - x_{k2} t_2 - \cdots - x_{kd} t_d\bigr),
\end{equation}
where $\Sigma$ is any triangulation of the simplicial polytope $P$.
Since (\ref{eq:momentgenf3}) does not depend on the triangulation 
$\Sigma$, so does the polynomial ${\rm Ad}_P$. It is
an invariant of the simplicial polytope $P$.

We refer to ${\rm Ad}_P$ as the {\em adjoint} of $P$. This polynomial
was introduced by Warren to study barycentric coordinates in
geometric modeling \cite{Warren1, Warren2}. He associates this
to the simple polytope $P^*$ dual to $P$. For simplicity, we  assume $0 \in {\rm int}(P)$.
The polytope $P^*$ is the set of points $(t_1,\ldots,t_d)$ for which 
all linear factors in (\ref{eq:momentgenf3}) and (\ref{eq:adjoint}) 
are nonnegative. This implies that ${\rm Ad}_P$ is nonnegative on $P^*$.
The main result in \cite{Warren1} states that the adjoint depends only on $P$,
 and not on its triangulation $\Sigma$.
For us, this is a corollary to Theorem~\ref{thm:nmgf_polytope}.
 
\begin{corollary}  \rm
The adjoint ${\rm Ad}_P$ is independent of the triangulation $\Sigma$ of the polytope $P$.
\end{corollary}

The $n$ linear factors in (\ref{eq:momentgenf2}),
(\ref{eq:momentgenf3}) and (\ref{eq:adjoint}) vanish on
the $n$ facets of the dual polytope $P^*$. 
This imposes interesting vanishing conditions on the adjoint ${\rm Ad}_P$.
A {\em non-face} of $P$ is any subset  $ \tau$ of $\{1,2,\ldots,n\}$ such that
 $\{ x_k \,: \,k \in \tau\}$ is not the vertex set of a face of $P$.
For any non-face $\tau$, we write
$L_\tau$ for the affine-linear space in $\RR^d$ that is
defined by the equations $\sum_{j= 1}^d x_{kj} t_j = 1$ for $k \in \tau$.
The collection of subspaces $L_\tau$ is denoted by 
$\mathcal{NF}(P)$. We call this the {\em non-face subspace arrangement} of
the simplicial polytope $P$. Equivalently, $\mathcal{NF}(P)$ is the set 
of all intersections in $\RR^d \backslash P^*$
of collections of facet hyperplanes of the simple polytope~$P^*$.

\begin{corollary}
The adjoint ${\rm Ad}_P$ is a polynomial of degree at most $n-d-1$ that vanishes on the non-face subspace arrangement $\mathcal{NF}(P)$.
\end{corollary}
 
\begin{proof}
The vanishing property follows from the fact that, for every non-face
$\tau$ of the polytope $P$, there exists a triangulation $\Sigma$ of $P$
that does not have $\tau$ as a face.
\end{proof}

In an earlier version of this article, 
we conjectured that, for every simplicial $d$-polytope $P$ with $n$ vertices,
the adjoint ${\rm Ad}_P$ is the unique polynomial 
of degree $n-d-1$ with constant term $1$ that  vanishes on the non-face subspace arrangement $\mathcal{NF}(P)$.
This is not quite true: 
For instance, if $P$ is a regular octahedron 
such that
its three diagonals intersect in a common point $(\delta_1, \delta_2, \delta_3) \in \RR^3$,
then the adjoint $A_P$ is $(1-\delta_1t_1 - \delta_2t_2 - \delta_3t_3)^2$ 
and the non-face subspace arrangement $\mathcal{NF}(P)$ consists of three lines in the plane defined by $\delta_1t_1 + \delta_2t_2 + \delta_3t_3 = 1$.
So there is not a unique quadratic polynomial vanishing along $\mathcal{NF}(P)$,
as every reducible quadratic polynomial with $(1-\delta_1t_1 - \delta_2t_2 - \delta_3t_3)$ as one of its two factors satisfies this vanishing property.
However, varying the vertices of $P$ without changing its combinatorial type makes the three lines in $\mathcal{NF}(P)$ skew
such that there is indeed a unique quadric surface passing through these three lines.
A corrected version of our conjecture was recently proven:

\begin{theorem}[see \cite{KR}] \label{thm:uniqueAdjoint}
Let $P$ be a $d$-polytope with $n$ vertices.
If the projective closure $\mathcal{H}_{P^\ast} \subset \PP^d$
of the hyperplane arrangement formed by the linear spans of the facets of the dual polytope $P^\ast$ is simple 
(i.e. through any point in $\PP^d$ pass at most $d$ hyperplanes in $\mathcal{H}_{P^\ast}$),
there is a unique hypersurface in $\PP^d$ of degree $n-d-1$ which vanishes along the projective closure of $\mathcal{NF}(P)$.
The defining polynomial of this hypersurface is the adjoint of $P$.
\end{theorem}

We note that the assumption in Theorem~\ref{thm:uniqueAdjoint} that the hyperplane arrangement $\mathcal{H}_{P^\ast}$ is simple implies that the polytope $P$ is simplicial. 
For instance, for a regular octahedron $P$, the plane arrangement $\mathcal{H}_{P^\ast}$ is not simple, but varying the vertices of $P$ makes $\mathcal{H}_{P^\ast}$ simple.

\medskip
The adjoint ${\rm Ad}_P$ is closely related to 
barycentric coordinates on the simple polytope $P^*$ and the associated
{\em Wachspress variety} in $\PP^{n-1}$; see \cite{IS, KR, Warren1, Warren2}.
These objects can be defined as follows.
Suppose that the origin $0$ lies in the interior of our simplicial polytope $P$,
and let $\Sigma_0$ be the triangulation of $P$
obtained by connecting $0$ to the boundary of $P$.
The facets of $\Sigma_0$ are $\sigma = 0 \cup \rho$ where 
$\rho$ is any facet of $P$. The formula (\ref{eq:adjoint})
holds for $\Sigma_0$, and we get
\begin{equation}
\label{eq:adjoint2} {\rm Ad}_P( t_1,t_2,\ldots,t_d) \,\,\,= \,\,\, 
\sum_{\rho\, {\rm is}\, {\rm a}\, {\rm facet} \, {\rm of} \, P} \beta_\rho 
\prod_{k \not\in \rho}\bigl(1- x_{k1} t_1 - \cdots - x_{kd} t_d\bigr).
\end{equation}
Here $\beta_\rho$ is the probability of the simplex $0 \cup \rho$, which is given by 
$|{\rm det}( x_k : k \in \rho)|$ divided by $d\, ! \, {\rm vol}(P)$. 
Each summand in (\ref{eq:adjoint2}) has degree $n-d$, but their 
sum has  degree $n-d-1$.

Let  $N$ denote the number of facets $\rho$ of $P$, i.e.~the number of vertices of $P^*$.
Consider the map $\, \RR^d \rightarrow \RR^N$
whose coordinates are the following rational functions, one for each~$\rho$:
$$
(t_1,\ldots,t_d)\,\, \mapsto\,\,
\frac{\beta_\rho \prod_{k \not\in \rho}\bigl(1- x_{k1} t_1  - \cdots - x_{kd} t_d\bigr)}{{\rm Ad}_P( t_1,t_2,\ldots,t_d) }.
$$
These are the {\em barycentric coordinates} of \cite{Warren1, Warren2}.
 These coordinate functions 
are nonnegative on $P^*$ and they sum up to $1$. The image of $P^*$ lies in the 
probability simplex with $N$ vertices.   We call this the {\em Wachspress model} of $P$.
The term  {\em model} is meant in the sense of algebraic statistics \cite{DSS}.
Its Zariski closure in  $\PP^{N-1}$ is the $d$-dimensional {\em Wachspress variety} of $P$.

In summary, the adjoint ${\rm Ad}_P$  was introduced in geometric modeling by
Warren \cite{Warren1}. It equals the numerator of the normalized 
moment generating function for the uniform
distribution $\mu_P$ on a simplicial polytope $P$ of type $\cP$.
The map $P \mapsto {\rm Ad}_P$  represents the computation of all moments
of $\mu_P$. This induces a polynomial map $X \mapsto {\rm Ad}_X$
on a dense open set of matrices $X \in \RR^{n \times d}$.
 Its image lies in an
affine space of dimension $\binom{n-1}{d} - 1$,
namely the space of polynomials of degree
$n-d-1$ in  $d$ variables with constant term $1$.
Passing to complex projective space, we define 
the {\em adjoint moment variety}
$\,\mathcal{M}_{\rm Ad}(\cP)\,$ to be the Zariski closure
of this image in $\PP^{\binom{n-1}{d}-1}$.
Readers of \cite{IS} are invited to regard $\mathcal{M}_{\rm Ad}(\cP)$ as a moduli space
of Wachspress varieties, and to contemplate the questions in Section~\ref{sec7}.

\section{One-Dimensional Moments}
\label{sec3}

In this section we characterize the relations among the
moments of the $1$-dimensional probability distributions that are obtained
by projecting the measures $\mu_P$ onto a line.
As before, let $P$ be a $d$-dimensional simplicial polytope with $n$
vertices. We fix the coordinate projection $\pi:  \RR^d \rightarrow \RR$
that takes $(t_1,t_2,\ldots,t_d)$ to its first coordinate $t_1$.
The pushforward $\pi_*(\mu_P)$ is a probability distribution on the
line $\RR^1$. 
The $i$th moment $m_i$ of this distribution equals the moment
$m_{i0\cdots 0}$ of $\mu_P$. For normalized moment generating functions, equation
(\ref{eq:momentgenf3}) implies
\begin{equation}
\label{eq:Aseries}
\sum_{i=0}^\infty \binom{d+i}{d} \,m_i t^i \quad = \quad
\frac{A_{n-d-1}(t)}{ (1- u_1 t)(1- u_2 t) \cdots (1-u_n t) } ,
\end{equation}
where $u_j = x_{j1}$ is the first coordinate of the $j$th vertex of the polytope $P$,
and the numerator is $A_{n-d-1}(t) =  {\rm Ad}_P (t,0,0,\ldots,0)$.
This is  a univariate polynomial of degree $n-d-1$.
We now confirm that the density of $\pi_*(\mu_P)$
is the polytopal spline mentioned in the Introduction.

\begin{proposition}
The density of $\,\pi_*(\mu_P)$ is a piecewise polynomial function
of degree $d-1$. Its value at any point $ a \in \RR^1$ is the
$(d-1)$-dimensional volume of the fiber $\pi^{-1}(a) \cap P$.
Moreover, this density function is $d-2$ times differentiable at its break points $u_1, \ldots, u_n$.
\end{proposition}

\begin{proof}
The pushforward $\pi_*(\mu_P)$ is the measure that assigns
to a segment $[v,w]$ in $\RR^1$ the nonnegative real number $\mu_P( \pi^{-1}([v,w]) \cap P)$.
This number is the probability that a uniformly chosen random point in 
the $d$-polytope $P$ has
its first coordinate between $v$ and $w$. That probability can be computed
by integrating the normalized $(d-1)$-dimensional volumes of $\pi^{-1}(a) $ for 
the scalars $a$  ranging from $v$ to $w$.
It is well-known in the theory of polyhedral splines
(cf.~\cite{DCoPro}) that
this volume (called the polytopal density) is a piecewise polynomial function of degree $d-1$ in the 
parameter $a$.  This spline function is polynomial  on each of the intervals 
$[u_i,u_{i+1}]$, and it is  $d-2$ times differentiable
at all its break points $u_i$.
\end{proof}

Fix any integer $r \geq 2n-d$ and consider the moments 
$m_0,m_1,\ldots,m_r$.
These correspond to the moments of $\mu_P$ whose index
set $\mathcal{A}$ equals $\{\!\{r \}\!\} = \{i e_1 : i=1,2,\ldots,r\}$.
Using the notation from the Introduction, we are interested in
the moment varieties $\mathcal{M}_{\{\!\{r\}\!\}}(\cP) \subset \PP^{r}$.

\begin{lemma}
The moment variety $\mathcal{M}_{\{\!\{r \}\!\}}(\cP)$ has dimension 
$2n-d-1$ in $\PP^{r}$. This variety depends only on $d$, $n$ and $r$.
It is independent  of  the combinatorial type~$\cP$ of the polytope.
\end{lemma}

\begin{proof}
Consider the probability distribution $\pi_*(\mu_P)$ where
$P$ runs over all polytopes of combinatorial type $\cP$.
Such a distribution is parametrized by
the $n$ parameters $u_i$ in the denominator of 
(\ref{eq:Aseries}) and the $n-d-1$ nonconstant coefficients of the 
numerator polynomial $A_{n-d-1}$.
Thus there are $2n-d-1$ degrees of freedom in specifying such a distribution,
or the associated spline function on $\RR^1$. Since the distribution
can be recovered from its first $2n-d$ moments (e.g.~by \cite{GLPR}),
the irreducible variety $\mathcal{M}_{\{\!\{r\}\!\}}(\cP)$ has dimension ${\rm min}(2n-d-1,r)$.

In the parametrization  above we obtain all polynomials
$A_{n-d-1}$ which are defined in some open set of the coefficient space $\RR^{n-d-1}$.
Hence the polytope type $\cP$  imposes only inequalities
but no equations on that parameter space. We therefore conclude that, 
for any combinatorial type $\cP$ of simplicial $d$-polytopes with $n$ vertices,
the moment variety $\mathcal{M}_{\{\!\{r\}\!\}}(\cP)$ is equal to the irreducible variety in $\PP^r$ that is given by the 
parametric representation (\ref{eq:Aseries}).
\end{proof}

We are now ready to state the main result in this section.
Our object of study is the subvariety 
$\mathcal{M}_{\{\!\{r\}\!\}}(d,n)$ of $\PP^r$ that is parametrically given by
 (\ref{eq:Aseries}), where $u_1,u_2,\ldots,u_n$ are
 arbitrary and $A_{n-d-1}(t)$ ranges over polynomials 
with constant coefficient $1$.
 We refer to this $(2n-d-1)$-dimensional variety as the
{\em $r$-th moment variety of polytopal measures of type $(d,n)$}.
To describe its homogeneous prime ideal, we introduce the normalized moments 
\begin{align*}
c_0 = c_1 = \cdots = c_{d-1} = 0 \quad\quad
\text{ and }  \quad\quad
c_{i+d} = \binom{d+i}{d} m_i \; \,\text{ for }\, i = 0,1, \ldots, r.
\end{align*}
We form
the following {\em Hankel matrix}
with $n+1$ rows and $r+d-n+1$ columns:
\begin{equation}
\label{eq:hankelmatrix}
\left(
\begin{array}{ccccccc}
c_0 & c_1 & \ldots & c_n & c_{n+1} & \ldots & c_{r+d-n} \\
c_1 & c_2 & \ldots & c_{n+1} & c_{n+2} & \ldots & c_{r+d-n+1} \\
\vdots & \vdots && \vdots & \vdots  && \vdots
\\
c_n & c_{n+1} & \ldots & c_{2n} & c_{2n+1} & \ldots & c_{r+d}
\end{array}
\right).
\end{equation}
Note that  each entry of this matrix is a scalar multiple of one of the moments $m_i$.

  \begin{theorem}\label{th:Kathlen}
  The homogeneous prime ideal in  $\mathbb{R}[m_0, m_1, \ldots, m_r]$ 
  that defines the moment variety $\mathcal{M}_{\{\!\{r\}\!\}}(d,n)$ 
    is generated by the maximal minors of the Hankel matrix (\ref{eq:hankelmatrix}).
    These minors form a reduced Gr\"obner basis with respect to any antidiagonal 
    term order, with initial monomial ideal $\langle m_{n-d},m_{n-d+1},\ldots,m_{r-n} \rangle^{n+1}$.
The degree of $\mathcal{M}_{\{\!\{r\}\!\}}(d,n)$ equals $\binom{r-n+d+1}{n}$.
\end{theorem}

The set-theoretic version of this theorem is implicit in the literature on polytopal moments
(cf.~\cite[Theorem~1]{GLPR}). We offer a proof based on results from commutative algebra.

\begin{proof}
Let $I$ be the ideal generated by the maximal minors of the matrix in (\ref{eq:hankelmatrix}).
The statement that $I$ is prime and has the expected codimension appears in \cite[Section 4A]{Eis}.
We fix the reverse lexicographic term order with
$m_0 > m_1 > \cdots > m_r$. The leading monomial of each maximal minor
of (\ref{eq:hankelmatrix}) is the product of the entries along the antidiagonal.
The ideal generated by all such antidiagonal products is the
$(n+1)$st power of the linear ideal $\langle m_{n-d},m_{n-d+1},\ldots,m_{r-n} \rangle$.
The codimension of that ideal equals the number $r-2n+d+1$
of occurring unknowns, and its degree is the number $\binom{r-n+d+1}{n}$
of monomials of degree $\leq n$ in these unknowns.
The Gr\"obner basis property for that term order follows from
\cite[Lemma 3.1]{Conca}. For an interesting refinement of that
Gr\"obner basis result see \cite[Corollary 3.9]{Nam}.

It remains to show that our moment variety $\mathcal{M}_{\{\!\{r\}\!\}}(d,n)$ equals the zero set of $I$.
Let $M(t)$ denote the formal power series on the left-hand side of (\ref{eq:Aseries}).
Fix a polynomial $\beta(t) = b_0 + b_1 t + b_2 t^2 + \cdots + b_n t^n$ with unknown
coefficients such that $\beta(t) M(t)$ is a polynomial of degree $n-d-1$. Hence the
coefficient of $t^i$ in $\beta(t) M (t)$ is zero for all integers $i \geq n-d$.
This constraint is a linear equation in $b = (b_n,b_{n-1},\ldots,b_1,b_0)$ whose
coefficients are the normalized moments $c_{j+d} = \binom{d+j}{j} m_j$. More precisely,
the equation for the coefficient of $t^i$ is
$$ b_n c_{i+d-n} + b_{n-1} c_{i+d-n+1} + \cdots + b_2 c_{i+d-2} + b_1 c_{i+d-1} + b_0 c_{i+d}  \,\,\, =\,\,\, 0 .$$
These equations for $i=n-d,n-d+1,\ldots,r$ are equivalent to the requirement that
the row vector $b$ is in the left kernel of the Hankel matrix (\ref{eq:hankelmatrix}).
Hence that matrix has rank $\leq n$ on  $\mathcal{M}_{\{\!\{r\}\!\}}(d,n)$. We conclude that 
  $\mathcal{M}_{\{\!\{r\}\!\}}(d,n)$ is contained in the variety of $I$.
We already saw that both are irreducible varieties of the same dimension.
Therefore, they are equal.
\end{proof}

\begin{remark} \rm
The recovery algorithm of \cite{GLPR} can be derived from the proof above.
For a given valid sequence of real moments, the Hankel matrix (\ref{eq:hankelmatrix})
has rank $n$. For such a matrix, we compute a generator
$b = (b_n, \ldots,b_1,b_0) $ of its left kernel.
The node points $u_1,\ldots,u_n$~are recovered as the roots of
$\beta(t) = \sum_{i=0}^n b_i t^i$. The numerator polynomial
in (\ref{eq:Aseries}) is found to be
$$ A_{n-d-1}(t) \quad = \quad \frac{1}{b_0} \sum_{\ell = 0}^{n-d-1}\biggl( \sum_{i=0}^{\ell} b_i c_{\ell+d-i} \biggr) \cdot t^\ell. $$
\end{remark}

It is  instructive to revisit Example \ref{ex:d1n2}  from the perspective of Theorem~\ref{th:Kathlen}.

\begin{example}[$d=1,n=2$] \rm \label{ex:einszwei}
The variety $\mathcal{M}_{\{\!\{r\}\!\}}(1,2)$ is the moment surface
in $\PP^r$ whose points represent the uniform probability distributions
on line segments in $\RR^1$. The prime ideal of this surface
is generated by the $3 \times 3$ minors of the $3 \times r$ Hankel matrix
\begin{equation}
\label{eq:3rows}
\left(\begin{array}{ccccccc} 
    0& m_0&2m_1&3m_2 & 4m_3 & \cdots & (r-1) m_{r-2}\\ 
    m_0& 2 m_1& 3 m_2 & 4 m_3 & 5 m_4 & \cdots & \quad\quad\;\;  r m_{r-1}\\
  2 m_1& 3 m_2& 4 m_3 & 5 m_4 &  6 m_5 &\cdots & (r+1) m_r\\
    \end{array}\right).
\end{equation}
These cubics form a Gr\"obner basis. The moment surface has degree $\binom{r}{2}$ in $\PP^r$.
Up to a factor of $4$, the leftmost $3 \times 3$ minor is equal to the cubic
 (\ref{eq:cubicmm})  whose surface is shown in Figure~\ref{fig:M13}.
\end{example}
   
 \section{Simplices}
 \label{sec4}
 
 In what follows we focus on the case $n=d+1$ when the
polytope $P$  is the $d$-simplex $\Delta_d$ with vertices
$x_k = (x_{k1},x_{k2}, \ldots,x_{kd})$ for $k=1,\ldots,d+1$.
From the normalized moment generating function in Lemma~\ref{lem:simplex}
we can derive the following explicit formula
for the moments of $\mu_{\Delta_d}$.

\begin{proposition}
\label{prop:explicitmI}
For $I = (i_1,\ldots,i_d) \in \NN^d$, the corresponding moment of the simplex equals
\begin{equation}
\label{eq:explicitmI} m_I(X) 
\,\, = \,\, \frac{i_1 ! \,i_2 !\, \cdots \, i_d ! \, d!}{(i_1{+}i_2{+}\cdots + i_d+d)!} \cdot \sum_{u}
\prod_{k=1}^{d+1} \frac{(u_{k1} {+} u_{k2} {+} \cdots {+} u_{kd})!}{u_{k1}! u_{k2}! \cdots u_{kd}!}
x_{k1}^{u_{k1}} x_{k2}^{u_{k2}} \cdots x_{kd}^{u_{kd}},
\end{equation}
where the sum is over nonnegative integer
$(d{+}1){ \times} d$   matrices $u$ with column sums given~by~$I$.
\end{proposition}

Proposition \ref{prop:explicitmI} shows that $m_I$ is a
fairly complicated polynomial of degree $|I|$
in the $d^2+d $ entries of $X=(x_1,\dots, x_{d+1})^T$.
However, these polynomials are still simpler than the rational functions
we obtain for moments of polytopes other than simplices.
For instance, consider the subalgebra of $\RR[X]$ generated
by all moments $m_I(X)$ in (\ref{eq:explicitmI}) where $I$ runs over $ \NN^d$.
We shall argue in Section \ref{sec7} that this is the algebra
of multisymmetric polynomials \cite{Dal}.

 In this section we are interested in the polynomial relations
among the moments $m_I$ where $I$ runs over an appropriate
finite subset $\mathcal{A}$ of $\NN^d \backslash \{0\}$. We homogenize
these relations with the special unknown $m_{00 \cdots 0}$ 
that represents the total  mass of the simplex.
This gives us homogeneous polynomial relations among the moments
indexed by $\mathcal{A} \cup \{0\}$. Their zero set in $\PP^{|A|}$
is the moment variety $\mathcal{M}_{\mathcal A}(\Delta_d)$.
The special case $d=1$ and $\mathcal{A} = \{\!\{r\}\!\}$,
where our variety is a surface,
was seen in Example~\ref{ex:einszwei}.

We next present an algorithm for recovering the 
$(d+1) \times d$ matrix $X =(x_{kl})$ from the above moments $m_I$ of order 
 $|I| \leq d+1$. There are $\binom{2d+1}{d}$ such moments $m_I$.
Let $\mathbb{L}$ denote the sum of all terms 
on the right-hand side in (\ref{eq:momentgenf}) where
$ 1 \leq i_1 + i_2 + \cdots + i_d \leq d+1$.
This is a polynomial in $t_1,t_2,\ldots,t_d$ with zero constant term.
We compute the formal inverse:
\begin{equation}
\label{eq:Minverse}
(1+ \mathbb{L})^{-1} \,\, = \,\, 1 - \mathbb{L} + \mathbb{L}^2 - \mathbb{L}^3 +
\mathbb{L}^4 +  \,\cdots\, + (-1)^{d+1} \mathbb{L}^{d+1}  \quad
{\rm mod} \,\,\bigl\langle t_1,t_2,\ldots,t_d \bigr\rangle^{d+2}.  
\end{equation}
Thus $(1+\mathbb{L})^{-1}$ is a polynomial of degree $\leq d+1$ in $t_1,t_2,\ldots,t_d$
with constant term $1$. This polynomial must  factor into linear factors,
one for each vertex of the desired simplex:
\begin{equation}
\label{eq:factorization}
(1+\mathbb{L})^{-1} \quad = \quad
\prod_{k=1}^{d+1} \bigl(\,1 \,-\, x_{k1} t_1 - x_{k2} t_2 - \cdots - x_{kd} t_d \bigr).
\end{equation}
A necessary and sufficient condition for such a factorization 
to exist is that the coefficients of $(1+\mathbb{L})^{-1}$ satisfy {\em Brill's equations} 
\cite{Dal, Guan}.
These classical equations  characterize  polynomials
that are products of linear factors, among all
polynomials of degree $\leq d+1$ in $d$ variables.                   
We write $[d+1]$ for the set of vectors $I \in \NN^d$ with $|I| \leq d+1$.
Our discussion implies:

\begin{corollary} \label{cor:brill}
Homogeneous equations that define
$\mathcal{M}_{[d+1]}(\Delta_d)$ set-theoretically
are obtained by substituting the polynomials in $m_I$
on the left-hand side of (\ref{eq:factorization}) into Brill's equations.
\end{corollary}
 
If we are given numerical values in $\QQ$ for the moments $m_I$
then the factorization (\ref{eq:factorization}) is found in exact
arithmetic by the built-in factorization methods in
any computer algebra system, provided the vertex
coordinates $x_{kl}$ of our simplex are rational numbers.
If the moments $m_I$ are rational but the $x_{kl}$ are not rational
then they are algebraic over $\mathbb{Q}$, and one can use 
algorithms for {\em absolute factorization} to obtain the
right-hand side of (\ref{eq:factorization}). If the moments are floating point
numbers then one uses tools from
{\em numerical algebraic geometry}
(e.g.~the software {\tt Bertini} \cite{Bertini}) to obtain
an accurate  factorization purely numerically.

We now return to the problem of computing the prime ideal
of our variety $\mathcal{M}_{[d+1]}(\Delta_d)$.
In practise, the method in Corollary \ref{cor:brill} did not
work so well. In what follows, we discuss some
 techniques that we found more effective in obtaining relations among moments.

In all  computations, it helps to use the fact that
the ideal of $\mathcal{M}_{\mathcal{A}}(\cP)$
is homogeneous with respect to a
natural $\ZZ^{d+1}$-grading. 
On the unknown moments this grading is given~by
\begin{equation}
\label{eq:finegrading} {\rm degree}(m_{i_1 i_2 \cdots i_d}) \,\, = \,\,
(1,i_1,i_2,\ldots,i_d). 
\end{equation}
This follows from the parametric representation 
of the moment varieties given in (\ref{eq:momentgenf}).
Our first result concerns the case $d=2$, i.e., the ideal of a moment variety for triangles.

\begin{proposition} \rm
\label{prop:trimonvar}
The triangle moment variety $\mathcal{M}_{[3]}(\Delta_2)$
has dimension $6$ and degree $30$. It lives in the projective space $ \PP^9$.
Its prime ideal  is minimally generated by 
eight quartics and one sextic. 
The degrees of the nine ideal generators in the 
$\ZZ^3$-grading given in (\ref{eq:finegrading}) are 
$$ (4,2,3), (4,3,2), (4,2,4), (4,3,3), (4,3,3), (4,4,2),  (4,3,4), (4,4,3), (6,6,6). $$
\end{proposition}

\begin{proof}
This computation was carried out with the technique of cumulants, to be introduced below.
For an explicit example, the  ideal generator of degree $(4,2,3)$   equals
\begin{equation}
\label{eq:rel423} \begin{matrix}
3 m_{02} m_{10}^2 m_{01} - 6 m_{11} m_{10} m_{01}^2 + 
3 m_{20} m_{01}^3-m_{03} m_{10}^2 m_{00}
+4 m_{11}^2 m_{01} m_{00} +m_{21} m_{02} m_{00}^2 \\
-\,4 m_{20} m_{02} m_{01} m_{00} 
+ 2 m_{12} m_{10} m_{01} m_{00} -m_{21} m_{01}^2 m_{00}
 +m_{03} m_{20} m_{00}^2-2  m_{12} m_{11} m_{00}^2 .
\end{matrix} \quad
\end{equation}
We shall present the derivation by means of {\tt Macaulay2} in the proof of Proposition~\ref{prop:grass}.
\end{proof}

Logarithms turn products into sums, and this can greatly
simplify calculations. To do this in the context of probability and
statistics, one passes from moments to {\em cumulants}.
Let $\mathbb{M}$ be the generating function on the right of (\ref{eq:momentgenf}).
The associated {\em normalized cumulant generating function} is defined
as the formal logarithm via $\,{\rm log}(1+x) = 
x- \frac{1}{2}x^2 + \frac{1}{3} x^3 - \cdots $:
\begin{equation}
\label{eq:cumulantgenf}
\mathbb{K} \,\,\, = \,\,\, {\rm log}(\mathbb{M}) \quad = \quad
\sum_{i_1, \ldots,i_d \in \NN} 
\!\!\! \frac{(i_1 {+} i_2 {+} \cdots {+} i_d -1)!}{i_1 ! \,\, i_2 !\,\, \cdots\,\, i_d! } k_{i_1 i_2 \cdots i_d} t_1^{i_1} t_2^{i_2} \cdots t_d^{i_d}.
\end{equation}
Here $k_{00 \cdots 0} = 0$. By comparing the coefficients of monomials $t^I$ in this identity,
we obtain the expressions for each cumulant $ k_I$ as a polynomial in the~moments 
$m_J$ where $|J| \leq |I|$.  

\begin{example}[$d=2$] \rm
\label{ex:M2cumu}
Here are the formulas for the cumulants $k_I$ of order $|I| \leq 3$ in terms
of the moments $m_J$ of order $|J| \leq |I|$, written
in the language of {\tt Macaulay2} \cite{M2}:
\label{ex:KintermsofM}
\begin{verbatim}
S = QQ[m30, m21, m12, m03, m20, m11, m02, m10, m01, m00];
k01=3*m01; k02=12*m02-9*m01^2; k03=27*m01^3+30*m03-54*m01*m02; k10=3*m10; 
k11=12*m11-9*m01*m10; k12 = -36*m01*m11-18*m10*m02+30*m12+27*m10*m01^2; 
k20 = 12*m20-9*m10^2; k21 = -18*m01*m20-36*m10*m11+30*m21+27*m01*m10^2; 
k30 = 30*m30+27*m10^3-54*m10*m20;
\end{verbatim}
We shall revisit this piece of code shortly, to represent the ideal generators in Proposition~\ref{prop:trimonvar}.
\end{example}

The transformation (\ref{eq:cumulantgenf}) from moments
to cumulants is easily invertible. Namely, the moment generating function is the
exponential of the cumulant generating function:
$$ \mathbb{M} \,\,\,\, = \,\,\,\, {\rm exp}(\mathbb{K})\, \,\,=\,\,\,
 1 \,+\, \mathbb{K} \,+\, \frac{1}{2} \mathbb{K}^2 \,+\, \frac{1}{6} \mathbb{K}^3 \,+\,
  \frac{1}{24} \mathbb{K}^4  \,+  \,\cdots. $$
 This identity expresses each moment  $ m_I$ as a polynomial in the cumulants $ k_J$ 
 where $|J| \leq |I|$.
     
     The factorial multipliers in the generating function (\ref{eq:cumulantgenf}) 
     are chosen  so that the normalized cumulants 
     $k_{i_1 i_2 \cdots i_d}$ of a simplex $\Delta_d$ coincide with the standard 
    {\em power sum multisymmetric functions} \cite[\S~1.2]{Dal}
     in its vertices $x_1,\ldots,x_{d+1}$. 
     This  is the content of the following corollary.

\begin{corollary}
\label{cor:cumu}
The cumulants of the uniform probability distribution on the simplex $\Delta_d$  are 
\begin{equation}
\label{eq:cumulantsymm}
k_{i_1 i_2 \cdots i_d} \quad = \quad \sum_{j=1}^{d+1}
x_{j 1}^{i_1} x_{j 2}^{i_2} \cdots x_{j d}^{i_d}.
\end{equation}
\end{corollary}    

\begin{proof}
Taking the logarithm of the left-hand side in (\ref{eq:momentgenf}), we
see that $\mathbb{K}$ is the sum of the expressions
$\, - {\rm log}\bigl(1 - x_{j 1} t_1 - x_{j 2} t_2 - \cdots - x_{j d} t_d \bigr)\,$
for $j=1,2,\ldots,d+1$. The coefficient of a non-constant
monomial $ t_1^{i_1} t_2^{i_2} \cdots t_d^{i_d}$ in 
the expansion of that expression equals
$\,x_{j 1}^{i_1} x_{j 2}^{i_2} \cdots x_{j d}^{i_d} $.
\end{proof}

\begin{remark} \rm
Both the moments (\ref{eq:explicitmI}) and the
cumulants (\ref{eq:cumulantsymm}) are multisymmetric functions in $x_1,x_2,\ldots,x_{d+1}$,
and they are expressible in terms of each other.  However, the
formula for the cumulants is much simpler than that for the moments.
For that reason, it seems advantageous to use cumulant coordinates
when studying the moment varieties of simplices.
\end{remark}

 Replacing moments with cumulants amounts to a change
 of coordinates in the affine~space
 $$ \mathbb{A}^{\binom{d+r}{d}-1} \quad = \quad \bigl\{ m_{00 \cdots 0} = 1 \bigr\} 
 \quad = \quad \bigl\{ k_{00 \cdots 0} = 0 \bigr\}. $$
This is the affine chart of interest inside the projective space 
 (\ref{eq:uptok}) which harbors $\mathcal{M}_{[r]}(\Delta_d)$.
 
 Ciliberto et al.~\cite{CCMRZ}  refer to this non-linear automorphism
    as a {\em Cremona linearization}.
In our situation, the Cremona linearization greatly simplifies the equations
that define $\mathcal{M}_{[r]}(\Delta_{d})$. We now illustrate this explicitly
for a simple case, namely for triangles ($d=2$) with $k=3$.
Here, Cremona linearization identifies our moment variety with a Grassmannian.

\begin{proposition} \label{prop:grass}
The restriction of the  $6$-dimensional triangle moment variety $\mathcal{M}_{[3]}(\Delta_2)$
in  $\PP^9$ to  the affine chart $\,\mathbb{A}^9 = \{m_{00} = 1 \} = \{k_{00} = 0\}$ 
can be identified with an affine chart of the $6$-dimensional Grassmannian of lines in $\PP^4$,
which  has its  Pl\"ucker embedding in~$\PP^9$.
\end{proposition}

\begin{proof}
 We give the identification with the Grassmannian as
 {\tt Macaulay2} code, starting from Example~\ref{ex:M2cumu}.
 The following ten expressions in the cumulants serve
as Pl\"ucker coordinates:
\begin{verbatim}
p01 = 3*k20-k10^2;
p02 = 6*k11-2*k10*k01;
p03 = 9*k21+12*k11*k10-5*k10^2*k01;
p04 = 18*k30-24*k20*k10+6*k10^3;
p12 = 3*k02-k01^2;
p13 = 9*k12-6*k11*k01+6*k02*k10-k10*k01^2;
p14 = 18*k21-12*k11*k10+12*k20*k01-2*k10^2*k01;
p23 = 9*k03-12*k02*k01+3*k01^3;
p24 = 18*k12+24*k11*k01-10*k10*k01^2;
p34 = 72*k21*k01+72*k12*k10+9*k20*k02-9*k20*k01^2
    -9*k11^2+18*k11*k10*k01-9*k02*k10^2-16*k10^2*k01^2;
\end{verbatim}
We next form the ideal generated by the five quadratic Pl\"ucker relations:
\begin{verbatim}
I = ideal(p01*p23-p02*p13+p03*p12, p01*p24-p02*p14+p04*p12,
 p01*p34-p03*p14+p04*p13, p02*p34-p03*p24+p04*p23, p12*p34-p13*p24+p14*p23);
\end{verbatim}
The ideal {\tt I} now contains five of the eight quartics in Proposition \ref{prop:trimonvar}
starting with that of degree $(4,2,3)$ in (\ref{eq:rel423}).
These five quartics generate the prime ideal of
the affine variety $\mathcal{M}_{[3]}(\Delta_2) \cap \mathbb{A}^9$.
To pass to the projective closure in $\PP^9$ we now homogenize and saturate:
\begin{verbatim}
J = saturate(homogenize(I,m00),m00);
toString mingens J
\end{verbatim}
This displays all nine minimal generators of the homogeneous prime ideal  of $\mathcal{M}_{[3]}(\Delta_2)$.
\end{proof}

Using the cumulant coordinates, it is possible to derive defining equations for
$\mathcal{M}_{[r]}(\Delta_d)$ for $r \geq d+2$ in terms of the equations for $r=d+1$.
This is done by the following technique:

\begin{proposition} \label{prop:eachcumu}
For the uniform probability distribution on the simplex $\Delta_d$,
each cumulant $k_I$ of order $|I| \geq d+2$  is a polynomial 
in the cumulants $k_J$ of order $|J| \leq d+1$.
\end{proposition}

\begin{proof}
We abbreviate $ X_k = x_{k1} t_1 + x_{k2} t_2 +  \cdots + x_{kd} t_d $ for $k=1,2,\ldots,d+1$.
For any $\ell \geq d+2$, we consider the power sum $X_1^\ell + X_2^\ell + \cdots + X_{d+1}^\ell$.
Using Newton's identities, we can write this uniquely as a polynomial $P_\ell$ with rational coefficients in
the first $d+1$ such power sums:
\begin{equation}
\label{eq:powersums}
 \sum_{k=1}^{d+1} X_k^\ell \,\,\,\, = \,\,\,\, P_\ell \biggl(\, \sum_{k=1}^{d+1} X_k^1\,,\,
\sum_{k=1}^{d+1} X_k^2\, ,\,\ldots, \sum_{k=1}^{d+1} X_k^{d+1} \biggr). 
\end{equation}
By Corollary~\ref{cor:cumu}, the left-hand side is the following polynomial of degree $\ell$ in $t_1,\ldots,t_d$:
$$ \sum_{k=1}^{d+1} X_k^\ell \,\,\,\, = \,\,\,\,  \sum_{I \,: \,|I| = \ell} \binom{|I|}{I}\, k_I \,t^I . $$
The same holds for the power sums occurring on the right-hand side of (\ref{eq:powersums}).
We expand the right-hand side and write it as a polynomial in $t_1,t_2,\ldots,t_d$.
Then each coefficient is a polynomial in the cumulants $k_J$ with $|J| \leq d+1$.
This gives the desired formula for $k_I$.
\end{proof}

\begin{example}[$d=2$] \rm 
The five fourth-order cumulants for a triangle in the plane $\RR^2$ admit the following 
polynomial expressions in terms of  the nine cumulants of lower order:
$$  \begin{matrix}
k_{04} &=&  k_{01}^4+3 k_{02}^2+4 k_{03} k_{01} - 6 k_{02} k_{01}^2, \\
k_{13} &=& -3 k_{02} k_{10} k_{01}+3 k_{11} k_{02}
+k_{03} k_{10}+ 3 k_{12} k_{01}-3 k_{11} k_{01}^2+k_{10} k_{01}^3 ,\\
k_{22} &=& k_{20} k_{02}  -k_{02} k_{10}^2+2 k_{11}^2+2 k_{12} k_{10}
+k_{10}^2 k_{01}^2+2 k_{21}  k_{01}-4 k_{11} k_{10} k_{01} -k_{20} k_{01}^2 , \\
k_{31} &=& k_{30} k_{01}+3 k_{20} k_{11}+3 k_{21} k_{10}
- 3 k_{11} k_{10}^2 - 3 k_{20} k_{10} k_{01} +k_{10}^3 k_{01}, \\
k_{40} &=& k_{10}^4+3 k_{20}^2+4 k_{30} k_{10}-6k_{20} k_{10}^2.
\end{matrix}
$$
These identities hold if we substitute
 $\,k_{ij} \, = \, x_{11}^i x_{12}^j \,+ \, x_{21}^i x_{22}^j \,+ \,x_{31}^i x_{32}^j $, so 
 they provide valid equations
 for $\mathcal{M}_{[4]}(\Delta_2)$ on the affine chart $\mathbb{A}^{14} = \{m_{00}  = 1\}$.
 To translate these equations into moment coordinates, we simply use the identities
 arising from $\, \mathbb{K} \, = \, {\rm log}(\mathbb{M})$, such as
$$ \begin{matrix} k_{04} &=& 60m_{04}- 72 m_{02}^2-81 m_{01}^4+216 m_{01}^2 m_{02} -120 m_{01} m_{03}, \\
k_{13} & = &\! 60 m_{13} {+} 108 m_{10} m_{01} m_{02} {-} 30 m_{10} m_{03}
{-} 81 m_{10} m_{01}^3 {+} 108 m_{01}^2 m_{11} {-} 90 m_{01} m_{12} {-} 72 m_{11} m_{02}.
\end{matrix}
$$
Consider the ideal generated by these polynomials in moments.
Just like in the end of the proof of Proposition \ref{prop:grass}, we 
homogenize and saturate  with respect to $m_{00}$. This yields
generators for the homogeneous prime ideal of
the triangle moment variety $\mathcal{M}_{[4]}(\Delta_2)$ in~$\PP^{14}$.
\end{example}

At this point, we note that all results in this section are valid
for configurations of $n \geq d+2$ points $x_1,x_2,\ldots,x_n$ in $\RR^d$,
but with the uniform measure on their convex hull replaced by a canonical polytopal measure.
Namely, consider the generating function on the left-hand side in (\ref{eq:momentgenf})
but with the upper index $n$ instead of $d+1$.
This is the normalized moment generating function for the probability measure
$\pi_*(\mu_{\Delta_{n-1}})$ on $\RR^d$ where $\pi$ denotes the canonical projection
from the simplex $\Delta_{n-1}$ onto the polytope $P  = {\rm conv}(x_1,x_2,\ldots,x_n)$:
$$
\prod_{k=1}^{n} \frac{1}{1 - (x_{k1} t_1 {+} x_{k2} t_2 {+} \cdots {+} x_{kd} t_d)} \,\,\, = \,
\sum_{i_1,i_2,\ldots,i_d \in \NN} \!\!\! \frac{(i_1 {+} i_2 {+} \cdots {+} i_d +n-1)!}{i_1 ! \, i_2 !\, \cdots\, i_d! \,(n-1) !}
m_{i_1 i_2 \cdots i_d} t_1^{i_1} t_2^{i_2} \cdots t_d^{t_d}. 
$$
The density function of
$\pi_*(\mu_{\Delta_{n-1}})$ is the canonical polytopal spline  supported on $P$.
This is piecewise polynomial of degree $n-d-1$ and differentiable of order $n-d-2$\; \cite{DCoPro}.

We consider the moments $m_I$ of order $|I | \leq r$ on the right-hand side above. These are polynomial functions
in the $nd$ unknowns $x_{kl}$.
Let  $\mathcal{I}_{r,d,n}$ denote the prime ideal of homogeneous polynomial relations
among these $\binom{r+d}{d}$ moments. For instance, 
the ideal $\mathcal{I}_{3,2,3}$ is the one with $9$ 
generators in $10$ unknowns seen  in Propositions \ref{prop:trimonvar}
and \ref{prop:grass}.

It would be interesting to compute the ideals $\mathcal{I}_{r,d,n}$
for as many values of $r$, $d$ and $n$ as possible, and to better understand their varieties.
For instance, the case $d=2$ and $n=4$
concerns the canonical piecewise linear densities on quadrilaterals in $\RR^2$.
It should be compared to the uniform distribution on  quadrilaterals,
to be studied in Section \ref{sec6}.

\smallskip

We conclude this section with a discussion of the tetrahedron $\Delta_3$. This has $12$ parameters,
namely the coordinates of the vertices $x_k = (x_{k1},x_{k2},x_{k3})$ for $k=1,2,3,4$.
We are interested in the moment variety $\mathcal{M}_{[3]}(\Delta_3)$ in 
$\PP^{19}$. Points on this variety represent  cubic surfaces in~$\PP^3$.
The coefficients of a cubic are specified by
the cumulants of the uniform distribution on $\Delta_3$:
$$ k_{ijl} \,\, = \,\,
x_{11}^i x_{12}^j x_{13}^l \,+\,
x_{21}^i x_{22}^j x_{23}^l \,+\,
x_{31}^i x_{32}^j x_{33}^l \,+\,
x_{41}^i x_{42}^j x_{43}^l 
\qquad \hbox{for $\,1 \leq i+j+l \leq 3 $.}
$$
We computed polynomials in the prime ideal of relations among these $19$ cumulants. 
This ideal is not homogeneous in the usual grading but it is homogeneous
in the $\ZZ^3$-grading given by ${\rm deg}(k_{ijl}) = (i,j,l)$. For a concrete example,
here is an ideal generator of degree $(3,2,2)$:
 \begin{tiny} $$ \begin{matrix}
k_{010}^2 k_{200} k_{102}+k_{100}^2 k_{020} k_{102}+k_{001}^2 k_{100} k_{110}^2+k_{010}^2 k_{100} k_{101}^2
-k_{100}^3 k_{002} k_{020}+k_{100}^2 k_{002} k_{120}+k_{100}^3 k_{011}^2+4 k_{110}^2 k_{102}
+2 k_{001} k_{100}^2 k_{020} k_{101} \\
+4  k_{101}^2 k_{120}+4 k_{011}^2 k_{300}
-2 k_{001} k_{010} k_{100} k_{101} k_{110}+2 k_{001} k_{010} k_{100} k_{011} k_{200}
-k_{010}^2 k_{100} k_{002} k_{200}-k_{001}^2 k_{100} k_{020} k_{200}-2 k_{010} k_{100}^2 k_{011} k_{101}
 \\ +2 k_{010} k_{100}^2 k_{002} k_{110}
 -2 k_{001} k_{100}^2 k_{011} k_{110} 
+10 k_{100} k_{011} k_{101} k_{110}+5 k_{100} k_{002} k_{020} k_{200}-2 k_{010} k_{100} k_{110} 
k_{102} 
+2 k_{010} k_{100} k_{101} k_{111}
\\ +2 k_{001} k_{100} k_{110} k_{111} 
-2 k_{001} k_{010} k_{200} k_{111} 
-2 k_{001} k_{100} k_{101} k_{120}+2 k_{010} k_{100} k_{011} k_{201}-2 k_{001} k_{100} k_{020} k_{201}
+2 k_{001} k_{010} k_{110} k_{201}\\ -2 k_{010} k_{100} k_{002} k_{210}
 +2 k_{001} k_{100} k_{011} k_{210}
+2 k_{001} k_{010} k_{101} k_{210}-2 k_{001} k_{010} k_{011} k_{300}+k_{001}^2 k_{200} k_{120}
-4 k_{002} k_{020} k_{300}+8 k_{002} k_{110} k_{210}\\ -8 k_{101} k_{110} k_{111} 
 -2 k_{100}^2 k_{011} k_{111}
-4 k_{002} k_{200} k_{120}+8 k_{011} k_{200} k_{111}-5 k_{100} k_{002} k_{110}^2-5 k_{100} k_{020} k_{101}^2 
+k_{010}^2 k_{002} k_{300} -4 k_{020} k_{200} k_{102} \\ -5 k_{100} k_{011}^2 k_{200}+8 k_{020} k_{101} k_{201}
-8 k_{011} k_{110} k_{201}-8 k_{011} k_{101} k_{210}+k_{001}^2 k_{020} k_{300}
-2 k_{001}^2 k_{110} k_{210}-2 k_{010}^2 k_{101} k_{201}.
\end{matrix}
$$ \end{tiny} 
Each relation among cumulants translates into a $\ZZ^4$-homogeneous relation among
the moments. The above polynomial translates into the
following polynomial of degree $(5,3,2,2)$:
\begin{tiny}
$$
\begin{matrix}
m_{000}^2 m_{002} m_{020} m_{300}-2 m_{000}^2 m_{002} m_{110} m_{210}
+m_{000}^2 m_{002} m_{120} m_{200}
-m_{000}^2 m_{011}^2 m_{300}+2 m_{000}^2 m_{011} m_{101} m_{210}
+2 m_{000}^2 m_{011} m_{110} m_{201} \\
-2 m_{000}^2 m_{011} m_{111} m_{200}-2 m_{000}^2 m_{020} m_{101} m_{201}+m_{000}^2 m_{020} 
m_{102} m_{200}
-m_{000}^2 m_{101}^2 m_{120}+2 m_{000}^2 m_{101} m_{110} m_{111} 
-m_{000}^2 m_{102} m_{110}^2 \\
 -2 m_{000} m_{001} m_{010} m_{110} m_{201}
 +2 m_{000} m_{001}^2 m_{110} m_{210} 
-m_{000} m_{001}^2 m_{120} m_{200}
+2 m_{000} m_{001} m_{010} m_{011} m_{300} -2 m_{000} m_{001} m_{010} m_{101} m_{210}
\\ -m_{000} m_{001}^2 m_{020} m_{300}
+2 m_{000} m_{001} m_{010} m_{111} m_{200}
-2 m_{000} m_{001} m_{011} m_{100} m_{210}+2 m_{000} m_{001} m_{020} m_{100} m_{201} 
+2 m_{000} m_{001} m_{100} m_{101} m_{120}
\\
-2 m_{000} m_{001} m_{100} m_{110} m_{111}
-m_{000} m_{002} m_{010}^2 m_{300}+2 m_{000} m_{002} m_{010} m_{100} m_{210}-5 m_{000} m_{002} m_{020} m_{100} m_{200} 
-m_{000} m_{002} m_{100}^2 m_{120}
\\ +2 m_{000} m_{010} m_{100} m_{102} m_{110}
+2 m_{000} m_{010}^2 m_{101} m_{201} 
-m_{000} m_{010}^2 m_{102} m_{200}-2 m_{000} m_{010} m_{011} m_{100} m_{201}
 -2 m_{000} m_{010} m_{100} m_{101} m_{111} \\
+5 m_{000} m_{002} m_{100} m_{110}^2
+5 m_{000} m_{011}^2 m_{100} m_{200}+2 m_{000} m_{011} m_{100}^2 m_{111}
-10 m_{000} m_{011} m_{100} m_{101} m_{110}-m_{000} m_{020} m_{100}^2 m_{102} \\
+5 m_{000} m_{020} m_{100} m_{101}^2 
+4 m_{001}^2 m_{020} m_{100} m_{200}-4 m_{001}^2 m_{100} m_{110}^2-8 m_{001} m_{010} m_{011} m_{100} m_{200}
+8 m_{001} m_{010} m_{100} m_{101} m_{110} \\ +8 m_{001} m_{011} m_{100}^2 m_{110}
 -8 m_{001} m_{020} m_{100}^2 m_{101} 
+4 m_{002} m_{010}^2 m_{100} m_{200}-8 m_{002} m_{010} m_{100}^2 m_{110}+4 m_{002} m_{020} m_{100}^3 \\
-4 m_{010}^2 m_{100} m_{101}^2+8 m_{010} m_{011} m_{100}^2 m_{101}-4 m_{011}^2 m_{100}^3.
\end{matrix} $$
\end{tiny}

Based on our computations, we propose the following conjecture.

\begin{conjecture}\label{conj:cummom}
Consider cumulants and moments of order $\leq 3$ for the uniform distribution
on a tetrahedron. They specify irreducible varieties of
dimension $12$ in $\mathbb{A}^{19}$ and $\PP^{19}$ 
respectively. The prime ideal for  cumulants has $44$ minimal generators. Their degrees are
\begin{tiny}
$$ 
\begin{matrix}
(223),(232),(322),
(134),(143),(314),(341),(413),(431),
(224),(224),(242),(242),(422),(422),
(332),(332),(332),(323),(323),(323),(233),\\(233),(233), 
(144),(414),(441),(333),(333),(225),(252),(522),
(234),(234),(243),(243),(324),(324),
(342),(342),(423),(423),(432),(432).
\end{matrix}
$$
\end{tiny}
The prime ideal for moments has $93$ minimal generators, namely
$90$ quintics and $3$ sextics.
\end{conjecture}

We shall return to the $90$ ideal generators of degree five in Proposition \ref{prop:tetra}.

\section{Symmetry and Invariants}
\label{sec5}

In this section we study the symmetries arising from the  group 
of affine transformations:
$$ {\rm Aff}_d \,\,:= \,\, \mathbb{R}^d \rtimes \mathrm{GL}_d(\RR). $$
This group is a subgroup of ${\rm GL}_{d+1}(\RR)$. It acts on column vectors $x = (x_1,\ldots,x_d)^T$ via
\begin{equation}
\label{eq:affineaction}
\begin{pmatrix}  x \\ 1 \end{pmatrix}
\,\,\mapsto\,\, \begin{pmatrix} A & b \\  0 & 1 \\ \end{pmatrix} \begin{pmatrix} x \\ 1 \\ \end{pmatrix}, 
\end{equation}
where $A = (a_{ij})$ is an invertible $d \times d$-matrix and
$b = (b_i)$ is a column vector in $\RR^d$.
This group acts naturally on the space of
realizations of a polytope type $\mathcal{P}$.
The action (\ref{eq:affineaction})
also induces an action on monomials and hence an action on moments $m_I$, $I \in \NN^d$.
Explicitly, 
\begin{equation}
\label{eq:affmom}
 m_{I} \,\,\mapsto \,\,\sum_{J} \nu_{IJ} \cdot m_{J}, 
 \end{equation}
where $\nu_{IJ} = \nu_{IJ}(A,b)$ is the coefficient of 
the monomial $x^J $ in the expansion of $(A x + b)^I$.
The sum in (\ref{eq:affmom}) is over all $J \in \NN^d$
such that $|J| \leq |I|$. Here are formulas for two simple~cases.

\begin{example}[$d=1$] \rm
The group ${\rm Aff}_1$ acts on the real line $\RR^1$ via $x \mapsto ax + b$,
where $a,b \in \RR$ with $a \not = 0$. Under this action, the $i$-th moment 
of a probability measure on $\RR^1$ is transformed into the following linear combination
of all moments of order at most $i$:
\begin{equation}
\label{eq:action1d}
 m_i \,\, \mapsto \,\, \sum_{j=0}^i \binom{i}{j} a^j b^{i-j} m_j. 
 \end{equation}
\end{example}

\begin{example}[$d=2$] \rm
The moments of order $\leq 2$ are the entries of
the symmetric matrix
$$ M \,= \, \begin{pmatrix}
m_{20} & m_{11} & m_{10} \\
m_{11} & m_{02} & m_{01} \\
m_{10} & m_{01} & m_{00} \end{pmatrix}.
$$
The upper left $2 \times 2$ block is the covariance matrix.
The group $ {\rm Aff}_2$ consists of $3 \times 3$ matrices
$$ 
Ab \,\, = \,\,
\begin{pmatrix}
a_{11} & a_{12} & b_1 \\
a_{21} & a_{22} & b_2 \\
0 & 0 & 1 \end{pmatrix}.
$$
With this matrix notation for $|I| \leq 2$, the action (\ref{eq:affineaction})   takes the form
$\, M \,\mapsto \, Ab \cdot M \cdot Ab^T$.
More generally, if we consider all moments of order $\leq r$ then
we can write these as an 
$r$-dimensional symmetric tensor of format $3 \times 3 \times \cdots \times 3$.
The action (\ref{eq:affineaction}) is given by multiplication of this tensor on all its $r$ sides
by the $3 \times 3$ matrix $Ab$. 
\end{example}

We have identified the space of moments of order $\leq r$ with the projective space
$ \PP^{\binom{d+r}{d}-1}$.
The formula (\ref{eq:affineaction}) defines a  linear action
of the group ${\rm Aff}_d$ on that projective space. 
Recall that, for each simplicial polytope
$P$ in $\RR^d$ and each subset $\mathcal{A} \subset \mathbb{N}^d$ with $0 \notin \mathcal{A}$, 
its associated moment variety $\mathcal{M}_{\mathcal{A}}(\cP)$
is a projective variety in $\PP^{|\mathcal{A}|}$.
In particular, if $\mathcal{A}$ is the index set
$\,[r]\, =\, \lbrace I \in \mathbb{N}^d	\mid 1 \leq |I| \leq r \rbrace$, 
then the moment variety $\mathcal{M}_{[r]}(\cP)$ lives in $ \PP^{\binom{d+r}{d}-1}$,
as in (\ref{eq:uptok}).

\begin{lemma}
\label{lem:actionOnMomentVarieties}
The moment variety $\mathcal{M}_{[r]}(\cP)$ of
a simplicial polytope $P$ in $\RR^d$ is invariant under the action of the group ${\rm Aff}_d$
of affine transformations on  the projective space $ \PP^{\binom{d+r}{d}-1}$.
\end{lemma}

\begin{proof}
The group ${\rm Aff}_d$ acts on $\PP^{\binom{d+r}{d}-1}$, and it
also acts on the space of all realizations $P$ of a given  combinatorial type $\cP$.
The map that takes a specific simplicial polytope $P$ to its point in the
variety $\mathcal{M}_{[r]}(\cP)$ is equivariant with respect to
the two actions, i.e.~the image of $P$ under an affine transformation
is mapped to the image of its moment vector under the same transformation.
This implies that $\mathcal{M}_{[r]}(\cP)$  is invariant under 
the action by ${\rm Aff}_d$.
\end{proof}


In cases where our moment variety is a hypersurface 
in $ \PP^{\binom{d+r}{d}-1}$,  its defining equation 
is a polynomial  that is invariant under ${\rm Aff}_d$.
It is thus of interest to study the {\em invariant ring}
\begin{equation}
\label{eq:invariantring}
 \RR \bigl[\, \,m_I \,: \,|I| \leq r\,\, \bigr]^{{\rm Aff}_d}. 
 \end{equation}
 Here and in what follows we use the term {\em invariant}  for
relative invariants, i.e. such that the transformation of an invariant polynomial equals the
original polynomial times a power of ${\rm det}(A)$. In other words,
an {\em invariant} is an absolute invariant of the subgroup ${\rm Aff}_d \cap {\rm SL}_{d+1}(\RR)$.

\begin{example}[$d=1,r=3$] \label{ex:invbinarycubic} \rm
The group ${\rm Aff}_1$ acts on
the polynomial ring $\RR[m_0,m_1,m_2,m_3]$ via (\ref{eq:action1d}).
The invariant ring has four generators, in degrees
$(1,1)$, $(2,2)$, $(3,3)$ and $(4,6)$:
$$
\begin{matrix}
a = m_0 \,, \quad
b = m_0 m_2-m_1^2 \, , \quad
c = m_0^2 m_3-3 m_0 m_1m_2+2 m_1^3, \\
d \,=\, m_0^2 m_3^2-6m_0 m_1 m_2 m_3+4 m_0 m_2^3+4 m_1^3 m_3-3 m_1^2 m_2^2.
\end{matrix}
$$
We verified the equality $\RR[m_0,m_1,m_2,m_3]^{{\rm Aff}_1} = \RR[a,b,c,d]$, see  Theorem \ref{thm:covariants}
below.
Note that $b$ and $d$ are the discriminants of binary forms of degree two and three.  The four invariants
satisfy the relation $ \,a^2 d - 4 b^3 - c^2 \,= \, 0$. This expresses the discriminant
$d$ in terms of the other three invariants on the affine open set $\{m_0 = 1\}$.
The invariant $c$ is the one of interest to us. The cubic surface it
defines in $\PP^3$ is given by (\ref{eq:cubicmm}) and shown in 
Figure \ref{fig:M13}.
\end{example}

Once we know generators for this invariant ring (\ref{eq:invariantring}),
 we can try to express our hypersurface
as a polynomial in these fundamental invariants.
Note that Hilbert's theorem on finite generation
does not directly apply  here, because
the group ${\rm Aff}_d$ is not reductive.
However, there is a nice method from
classical invariant theory using {\em covariants}, which allows us to conclude
finite generation and to compute the invariants of ${\rm Aff}_d$ we are interested in.

Set $V = \RR^{d+1}$, with standard basis denoted by $\{u_1,u_2,\ldots,u_{d+1}\}$. 
We identify the symmetric power $S_r(V)$ with our
space of moments $m_I$ of order $|I|$ at most $r$.
The action of the general linear group $G = {\rm GL}_{d+1}(\RR)$ on
   $S_r(V)$ restricts to the action of the affine group ${\rm Aff}_d$ on moments.
 The group $G$ acts naturally on the dual space $V^*$. 
 We consider the action of $G$ on the direct sum
$S_r (V) \oplus V^*$, and the induced action on the polynomial ring
\begin{equation}
\label{eq:SkVVdual}
 \RR[ S_r(V) \oplus V^*] \,\,\,= \,\,\,
\RR \bigl[ m_I \,: \,|I| \leq r \bigr] \,\otimes_\RR \, \RR[ u_1,u_2,\ldots,u_{d+1}\,]. 
\end{equation}
A $G$-invariant in this polynomial ring is known as a {\em covariant}.
Thus $\, \RR[ S_r(V) \oplus V^*]^G \,$ is the ring of covariants of
$S_r(V)$. This ring is finitely generated because $G$ is reductive.

Let $\psi$ be the ring epimorphism $\,\RR[S_r(V) \oplus V^*] \,\rightarrow \RR[ S_r(V) ]\,$
defined by $u_{d+1} \mapsto 1$ and $u_i \mapsto 0$ for $i =1,2,\ldots,d$.
This reflects the special role played by the last index 
in the realization of ${\rm Aff}_d$ as a subgroup of $G$.
The following result is known in classical invariant~theory.

\begin{theorem} \label{thm:covariants}
The map $\psi$ induces an isomorphism between covariants and affine invariants:
\begin{equation}
\label{eq:covariants} \RR[\,S_r(V) \oplus V^*\,]^G \,\,\simeq \,\, \RR[ S_r(V) ]^{{\rm Aff}_d}. 
\end{equation}
\end{theorem}

\begin{proof}
This statement is a special case of~\cite[Theorem~11.7]{grosshans}. 
\end{proof}

The  basic covariant is the homogeneous polynomial  itself. In our notation,
\begin{equation}
\label{eq:inournotation}
 f \,=\, f(m,u) \,\,\,\, = \,\,\, \sum_{|I| \leq r}  {\scriptstyle \binom{r}{I, r-|I|}}\cdot m_I \cdot u^I u_{d+1}^{r-|I|}. 
 \end{equation}
The image of $f$ under the isomorphism (\ref{eq:covariants}) is the
 ${\rm Aff}_d$-invariant moment coordinate $$\psi(f) = m_{00 \cdots 0}.$$
The {\em degree} of a covariant $g = g(m,u)$ is its degree in the unknowns $m_I$.
The {\em order} of $g$ is its degree in the unknowns $u_j$.
The form $f$ is a covariant of degree $1$ and order $r$.
Covariants of order $0$ are invariants of $G$. The degree of an affine invariant 
in the $\ZZ^{d+1}$-grading (\ref{eq:finegrading}) can be
read off from the degree and the order of the corresponding covariant:

\begin{lemma}
Let $g$ be a covariant of degree $p$ and order $o$ for the space $S_r(\RR^{d+1})$ 
of degree $r$ forms. Then the integer $rp-o$ is a positive multiple of 
the number $d+1$ of unknowns.  Setting $q =\frac{rp-o}{d+1}$,
the degree of the associated affine invariant $\psi(g)$ equals
$(p,q,q,\ldots,q)$.
\end{lemma} 

\begin{proof} 
Consider the diagonal matrix ${\rm diag}(t,t,\ldots,t)$
in $G = {\rm GL}_{d+1}(\RR)$. It acts on $S_r(V)$ 
by mutiplying the vector $m$ with $t^r$. It acts on $V^*$
by multiplying the vector $u$ with $t^{-1}$. The covariant $g(m,u)$ of
degree $p$ and order $o$ is transformed  by the action of this diagonal matrix into
$g(t^r m, t^{-1} u) = t^{pr-o} g(m,u)$.
The multiplier $t^{pr-o}$ is a power of
${\rm det}(A) = t^{d+1}$, so $q = \frac{rp-o}{d+1}$ is an integer.
It follows that $ {\rm diag}(t_1,t_2,\ldots,t_{d+1})$
takes $g(m,u)$ to $t_1^q t_2^q \cdots t_{d+1}^q g(m,u)$.
We find that  $\psi(g)(m) = g(m,e_{d+1})$ has degree $(p,q,\ldots,q)$
with respect to the grading in (\ref{eq:finegrading}).
\end{proof}

\begin{example}[$d=1,r=3$] \rm
We derive Example \ref{ex:invbinarycubic} from the
classically known covariants of the binary cubic. The four generators 
of the ring $\RR[m_0,m_1,m_2,m_3,u_1,u_2]^G$ are
\begin{itemize}
\item the binary cubic  $ A$ itself, of degree $1$ and order $3$, \vspace{-0.1in}
\item the Hessian $ B$, of degree $2$ and order $2$, \vspace{-0.1in}
\item the Jacobian of $ A$ and $ B$, denoted by $ C$, of degree $3$ and order $3$, \vspace{-0.1in}
\item the discriminant $ D$, of degree $4$ and order $0$.
\end{itemize}
Applying $\psi$ to these covariants yields the corresponding affine invariants in
Example~\ref{ex:invbinarycubic}.
\end{example}

\begin{example}[$d=2,r=3$] \label{ex:covv33} \rm
Consider any  probability measure on $\RR^2$.
Its moments of order $\leq 3$  can be encoded as the coefficients of a ternary cubic
\begin{equation}
\label{eq:ternarycubic}
\begin{matrix} f &  = &
m_{30}u_1^3+3 m_{21} u_1^2 u_2 + 3 m_{20} u_1^2 u_3
+3 m_{12} u_1 u_2^2 + 6 m_{11} u_1 u_2 u_3 \\ & & +3 m_{10} u_1 u_3^2
+m_{03} u_2^3 + 3 m_{02} u_2^2 u_3
+3 m_{01} u_2 u_3^2 + m_{00} u_3^3.
\end{matrix}
\end{equation}
The notation is as in (\ref{eq:inournotation}).
It is  classically known that $f$ has six fundamental covariants:
\begin{equation}
\label{ex:covariants33}
\begin{matrix}
{\rm covariant} &  f &  S & T &   H & G & J \\
({\rm degree}, {\rm order} )& (1,3) & (4,0) & (6,0) & (3,3) & (8,6) & (12,9) \\
\end{matrix}
\end{equation}
First is the ternary cubic $f$ itself, of degree $1$ and order $3$.
Next are the  {\em Aronhold invariants} $S$ and $T$, of degree $4$ and $6$ resp. 
These are followed
by the Hessian $H$.
The covariant $G$ is explained in Dolgachev's book \cite[Section 3.4.3]{Dolgachev},
where the following formula can be  found:
$$ \qquad G \,\, = \,\, 
{\rm det} \begin{small} \begin{pmatrix} f_{11} & f_{12} & f_{13} & h_1 \\
f_{12} & f_{22} & f_{23} & h_2 \\
f_{13} & f_{23} & f_{33} & h_3 \\
h_1 & h_2 & h_3 & 0 \end{pmatrix}  \end{small}
\qquad \hbox{with}\,\, \text{
$f_{ij} = \frac{ \partial^2 f}{ \partial u_i \partial u_j} \,\,$
and $\,\,h_i = \frac{\partial H}{\partial u_i}$}. $$
The last covariant $J$ is the Jacobian of $f$, $H$ and $G$.
This is known as the {\em Brioschi covariant}.

The six fundamental affine invariants are the images of the fundamental covariants 
under replacing $(u_1,u_2,u_3) $ with $ (0,0,1)$:
\begin{equation}
\label{ex:affinv33}
\begin{matrix}
 \text{affine invariant} & m_{00} =\psi(f) &  s = S &  t = T & h = \psi(H) &  g = \psi(G) & j = \psi(J) \\
 \text{$\ZZ^3$-degree} & (1,0,0) &\, (4,4,4)\,  & \,(6,6,6)\, & (3,2,2) & (8,6,6) & (12,9,9) \\
 \text{\# terms} &  1 & 25 & 103 & 5 & 168 &  892 \\
 \end{matrix}
 \end{equation}
 \end{example}

We summarize our derivation of the affine invariants of ternary cubics as follows:

\begin{proposition} \label{prop:moliensuccess}
For $d=2$ and $r=3$, the invariant ring (\ref{eq:invariantring})
equals $\RR[m_{00},s,t,h,g,j] $ modulo one
homogeneous relation of degree
$(24,18,18)$. Hence its Hilbert series equals
\begin{equation}
\label{eq:hilbertseries}
\begin{small}
\frac{1 + x^{12} y^9 z^9}{
(1-x)(1-x^4 y^4 z^4) (1- x^6 y^6 z^6)(1-x^3 y^2 z^2) (1-x^8 y^6 z^6) (1-x^{12} y^9 z^9)} . \end{small}
\end{equation}
\end{proposition}

The moment varieties $\mathcal{M}_{[r]}(\cP)$ are hypersurfaces in only very few cases.
Examples include
$\,\cP =$  quadrilateral with $r=3$, 
$\cP = $ 13-gon with $r=6$, \ or $\cP = $ octahedron with $r=3$.
In those cases there is a single affine invariant.
The first one is featured in the next section.

\section{Quadrilaterals and Beyond}
\label{sec6}

This section is devoted to the smallest non-simplex. Let $Q$ be a quadrilateral in the plane. 
Some of its moments were already explicitly shown in Example~\ref{ex:quadrMoments}.
We know the normalized moment generating function  from Section~\ref{sec2}.
The only non-faces of the quadrilateral $Q$ are its two diagonals. 
Hence, the adjoint  ${\rm Ad}_Q$ is given by the intersection point of these diagonals. 
More specifically, if $x_1, x_2, x_3, x_4$ denote the cyclically labeled vertices of $Q$ and $(\delta_1, \delta_2)$ is the diagonal intersection point, then the normalized moment generating function of $Q$ equals 
\begin{equation}
\label{eq:quadformula}
\frac{1-\delta_1 t_1-\delta_2 t_2}{(1-x_{11}t_1 - x_{12}t_2)(1-x_{21}t_1 - x_{22}t_2)
(1-x_{31}t_1 - x_{32}t_2)(1-x_{41}t_1 - x_{42}t_2)}.
\end{equation}

It is a non-trivial task to compute relations among the moments of quadrilaterals.
The easiest relations are given by Theorem~\ref{th:Kathlen}, if we take the Hankel matrix~\eqref{eq:hankelmatrix} for $r=6$.

\begin{example} \rm
\label{eq:quadonline}
Consider the moments $m_{i0}$ where $i =0,1,\ldots,6$.
The corresponding moment variety
$\mathcal{M}_{\{ \! \{ 6 \} \! \}}(Q) $ is the
hypersurface $\mathcal{M}_{\{ \! \{ 6 \} \!\}}(2,4) \subset \PP^6$. It is defined by
the determinant of
\begin{align}
\label{eq:HankelQuadr}
\left( \begin{array}{ccccc}
0 & 0 & m_{00} & 3 m_{10} & 6 m_{20} \\
0 & m_{00} & 3 m_{10} & 6 m_{20} & 10 m_{30} \\
m_{00} & 3 m_{10} & 6 m_{20} & 10 m_{30} & 15 m_{40} \\
3 m_{10} & 6 m_{20} & 10 m_{30} & 15 m_{40} & 21 m_{50} \\
 6 m_{20} & 10 m_{30} & 15 m_{40} & 21 m_{50} & 28 m_{60}
\end{array} \right).
\end{align}
This relates the moments
 of the pushforward measure given from projecting $Q$ onto a line.
\end{example}

What we are actually interested in are \emph{mixed relations}, i.e.~equations 
in the moments $m_{ij}$  that do not come from projections onto lines as in 
Example~\ref{eq:quadonline}.
The dimension of $\mathcal{M}_{\mathcal{A}}(Q)$ in $\mathbb{P}^{|\mathcal{A}|}$ is eight if 
$\mathcal{A} \subset \mathbb{N}^2$ is big enough.
We first show an interesting scenario with $|\mathcal{A}| = 8$.

\begin{example}\rm
Let $\mathcal{A} :=( \lbrace 0,1,2 \rbrace \times \lbrace 0,1,2 \rbrace)
\backslash \lbrace (0,0) \rbrace $.
These moments are algebraically independent. Hence the moment variety
$\mathcal{M}_{\mathcal{A}}(Q) $ is equal to the ambient space $ \PP^8$. Consider
the map $ (\mathbb{C}^2)^4 \dashrightarrow\PP^8$
which sends quadrilaterals to their moments in $\mathcal{A}$.
A computation with the software {\tt HomotopyContinuation.jl} \cite{homotopyCont}
reveals that randomly chosen fibers of this map consist of $80$ points over $\CC$.
We conclude that the map $ (\mathbb{C}^2)^4 \dashrightarrow \PP^8$
is generically $80$-to-$1$. The dihedral group of order $8$
acts on each fiber by permuting vertices of $Q$.
Hence each fiber consists of $10$ geometric configurations, generally over~$\CC$.

For a concrete example, consider the quadrilateral
$X = \{(1,-1), (3,2), (2,4), (-1,2)\}$. The fiber for this $X$
consists of $80$ \emph{real} points. 
These correspond to four non-convex quadrilaterals, two convex quadrilaterals and 
four quadrilaterals with self-crossings; see Figure~\ref{fig:quads}.

\begin{figure}[h]
\centering
\includegraphics[width=0.24\textwidth]{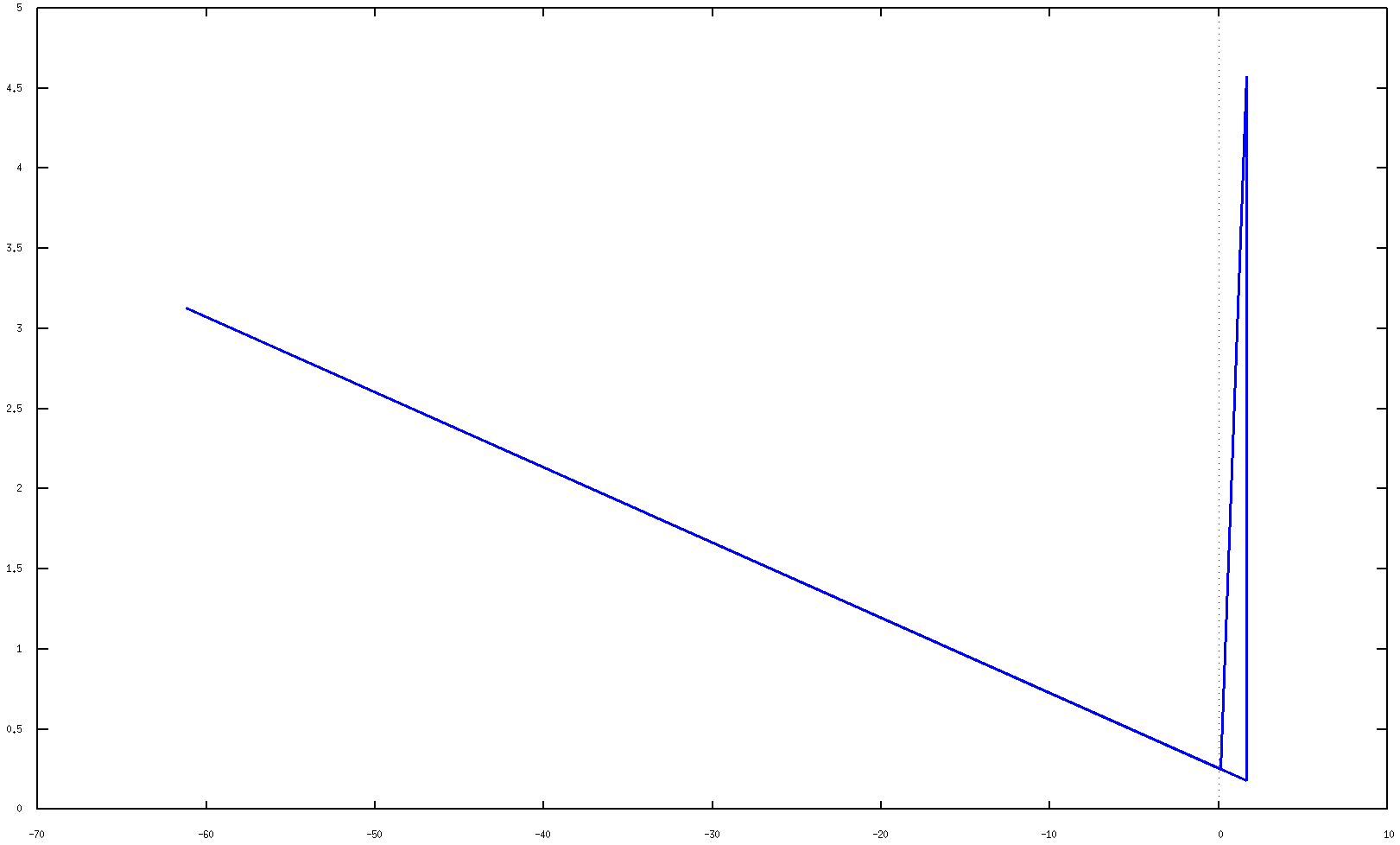} 
\includegraphics[width=0.24\textwidth]{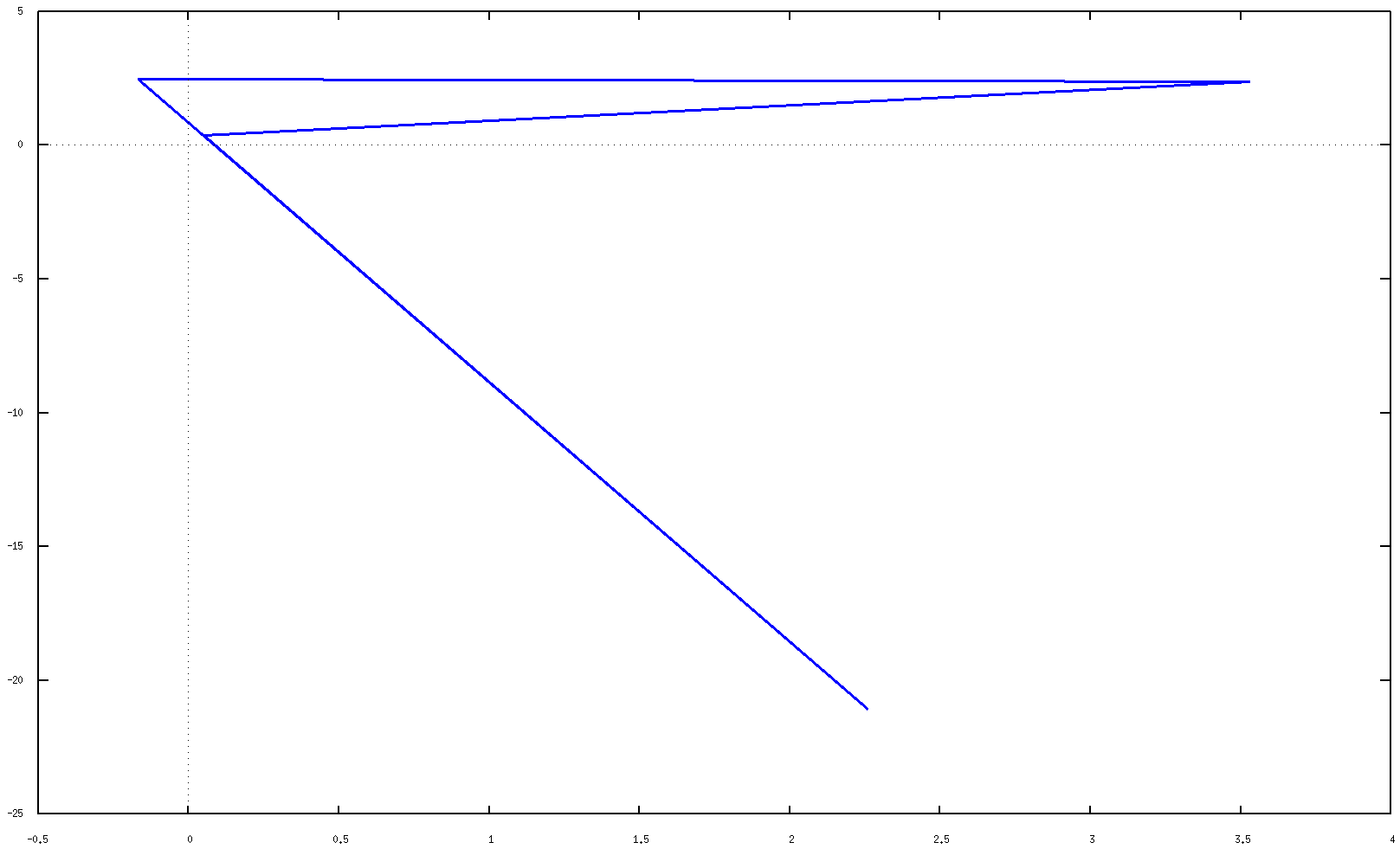} 
\includegraphics[width=0.24\textwidth]{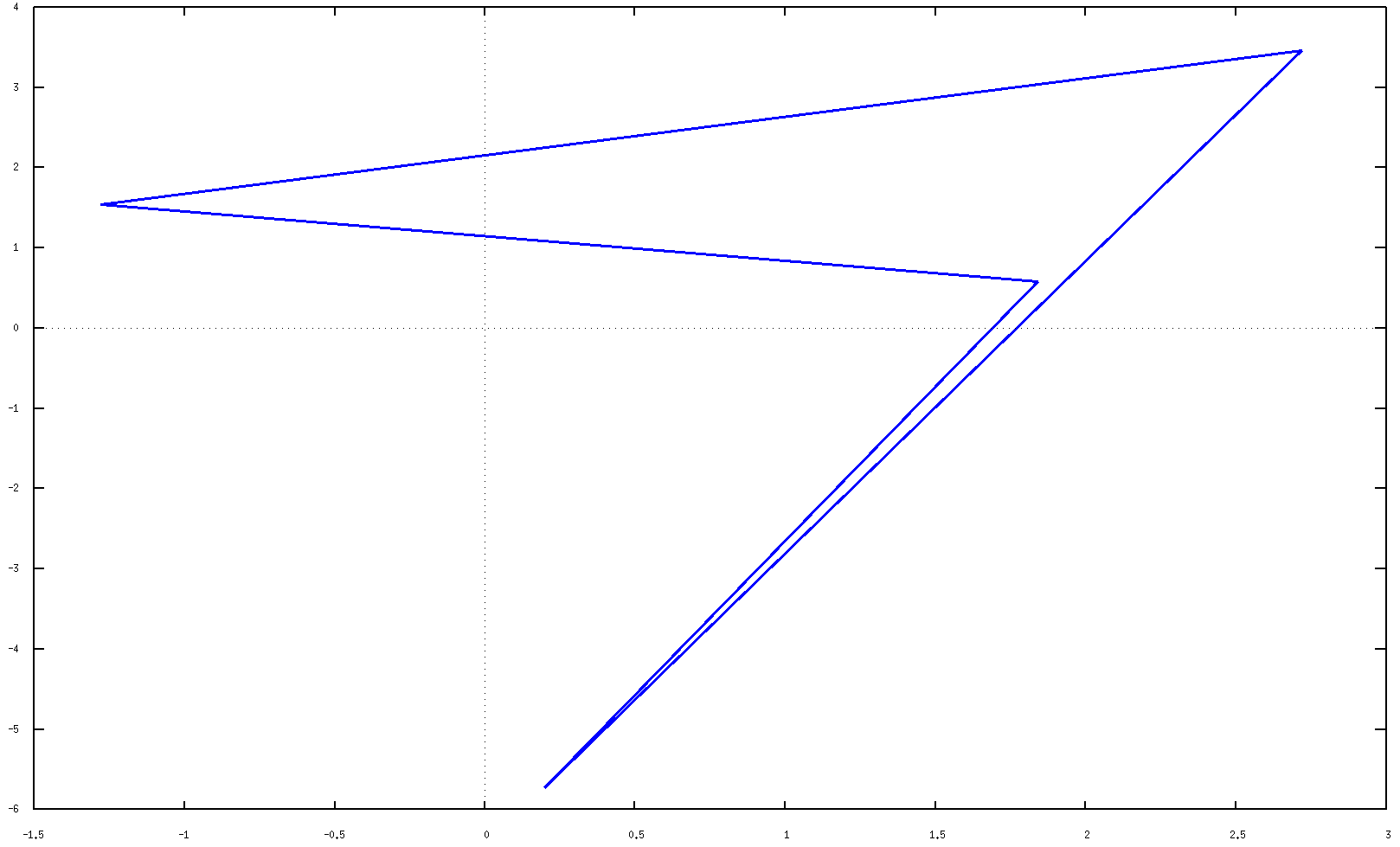} 
\includegraphics[width=0.24\textwidth]{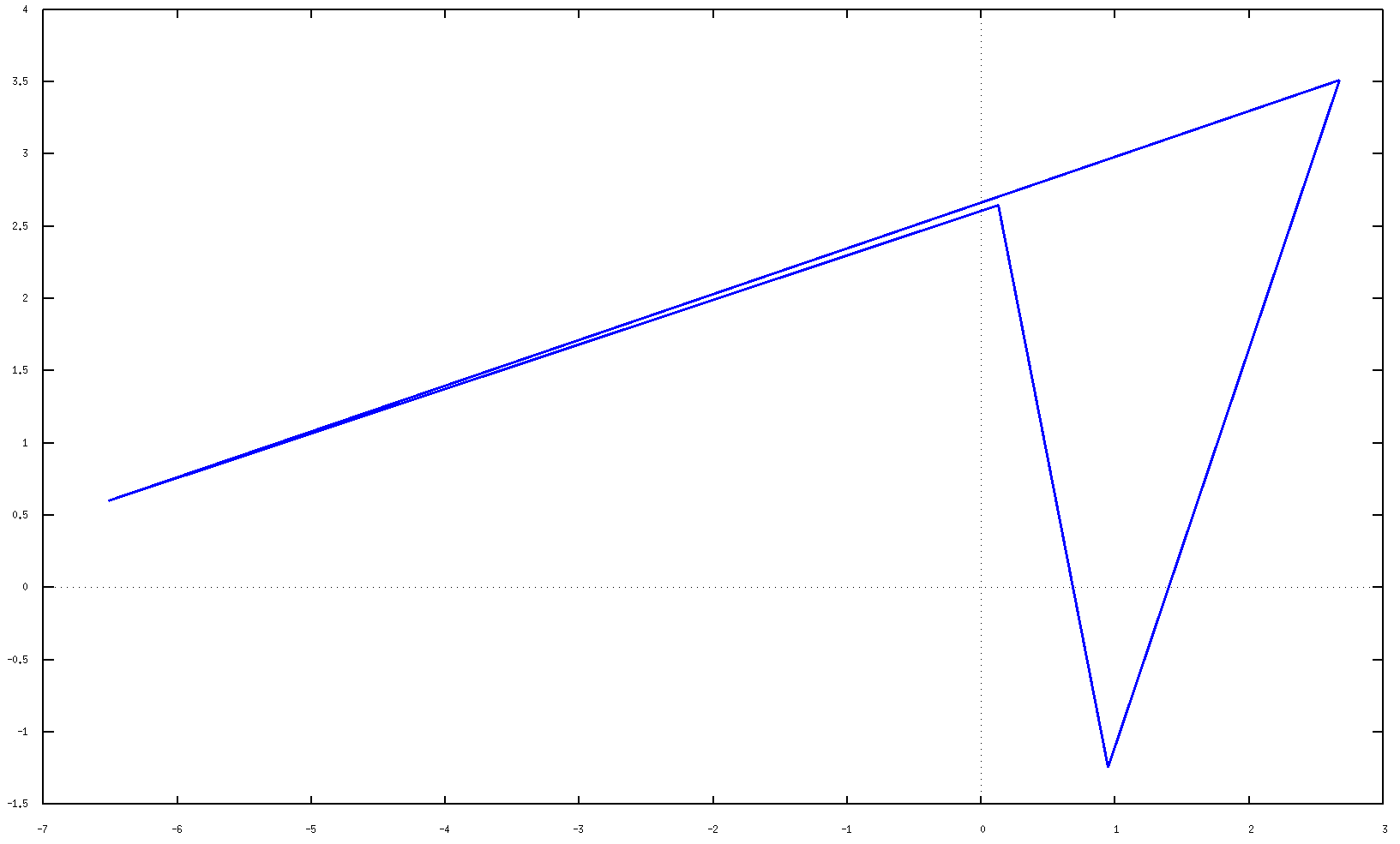} \\
\includegraphics[width=0.29\textwidth]{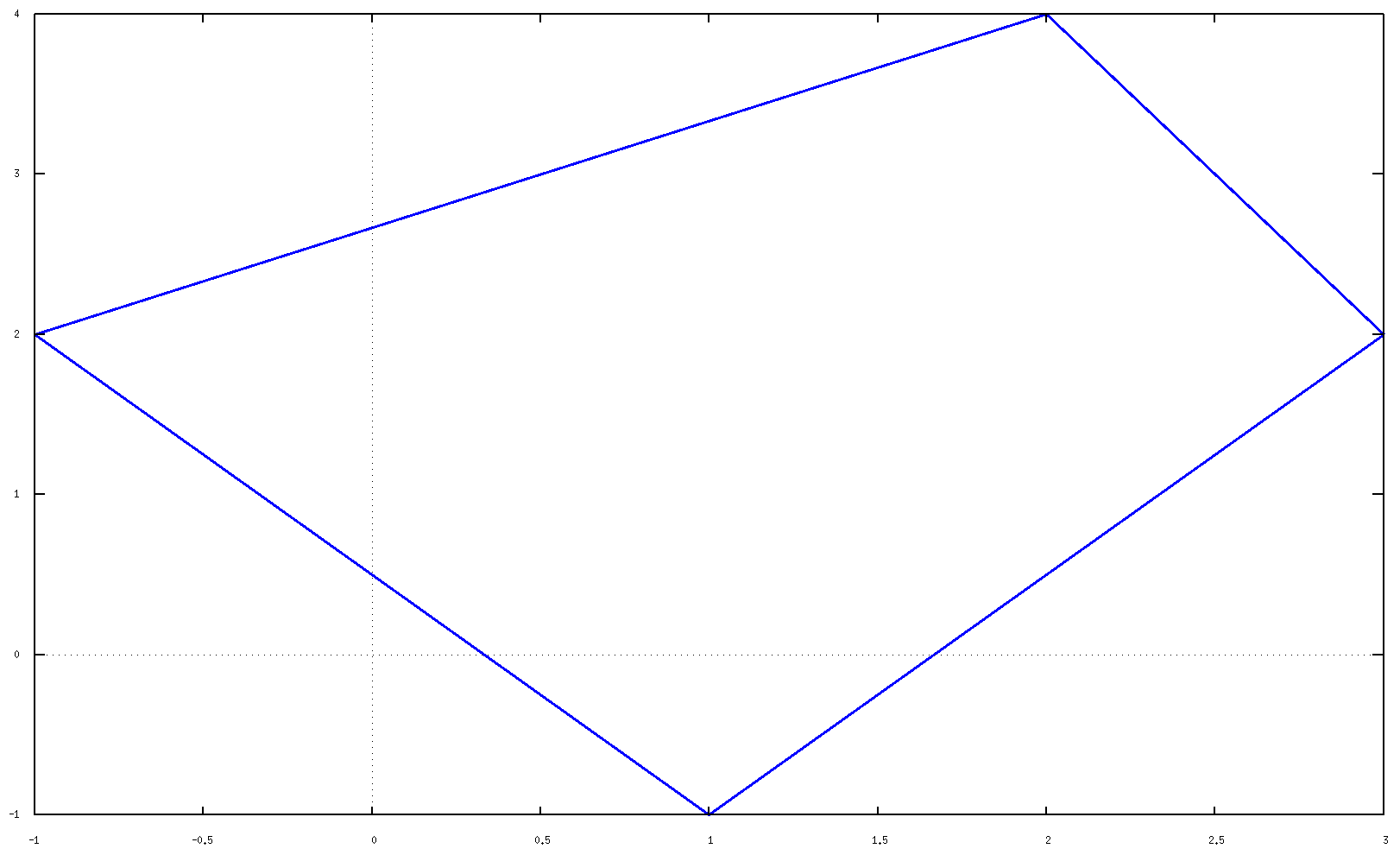} \qquad
\includegraphics[width=0.29\textwidth]{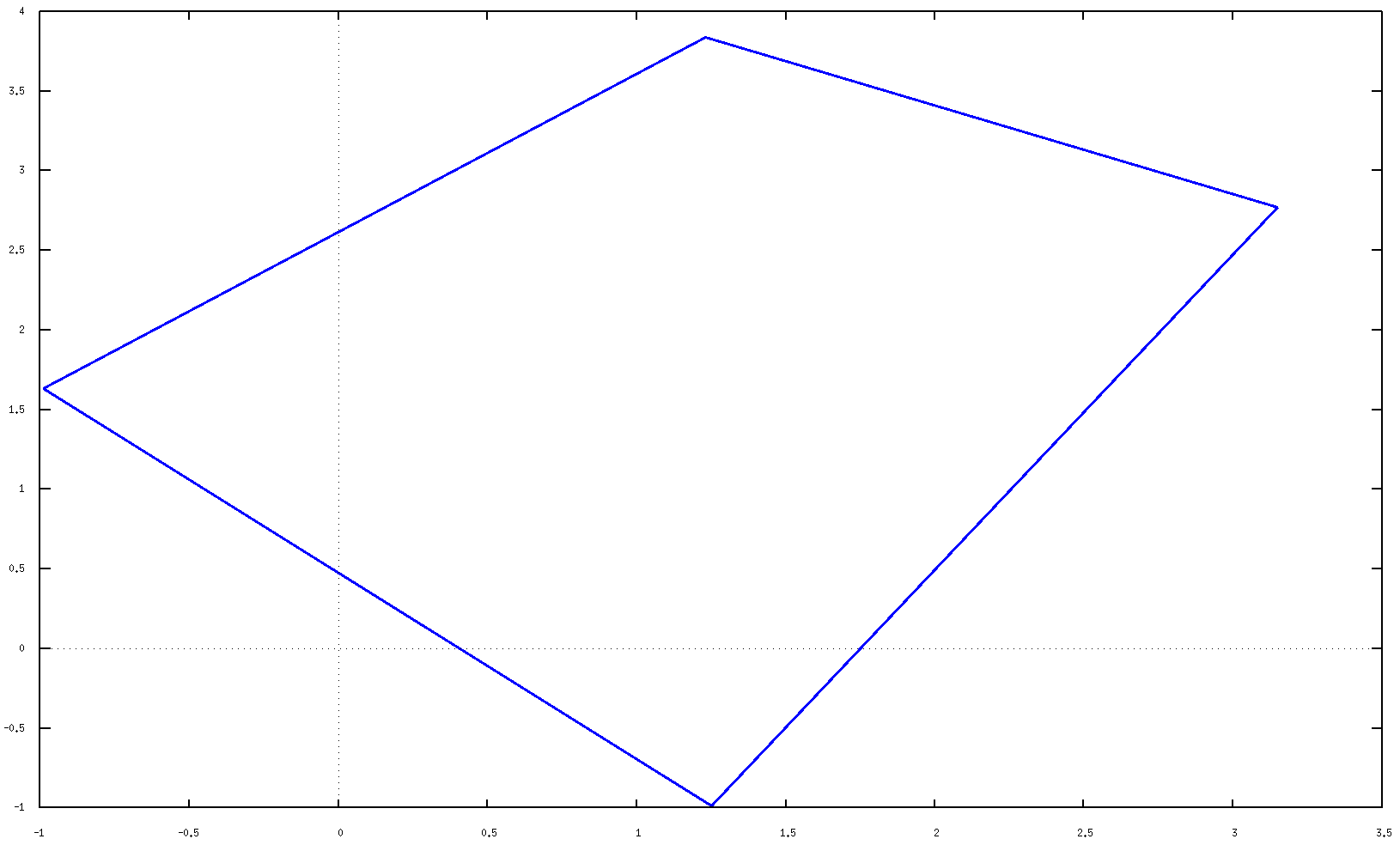} \\
\includegraphics[width=0.24\textwidth]{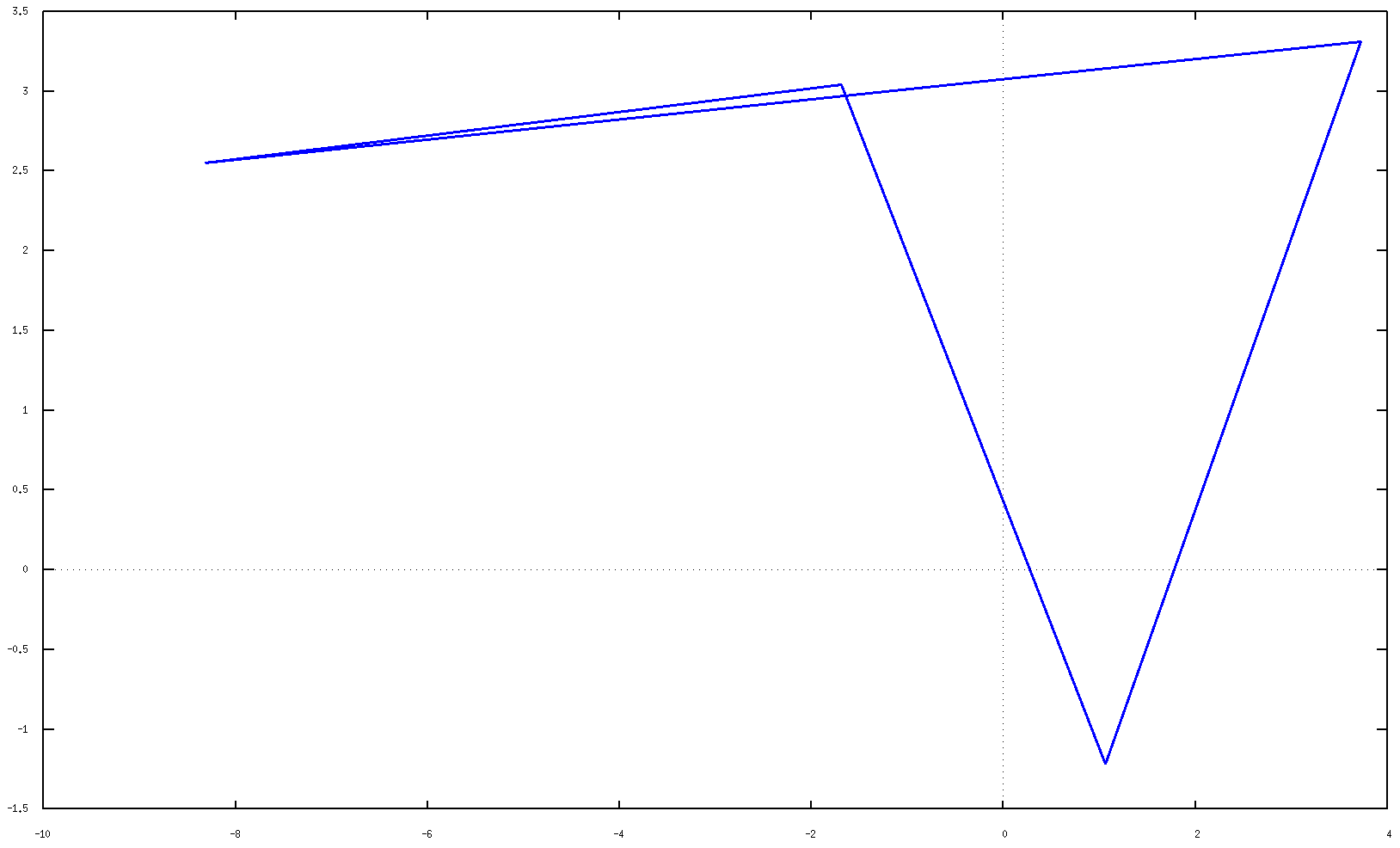}
\includegraphics[width=0.24\textwidth]{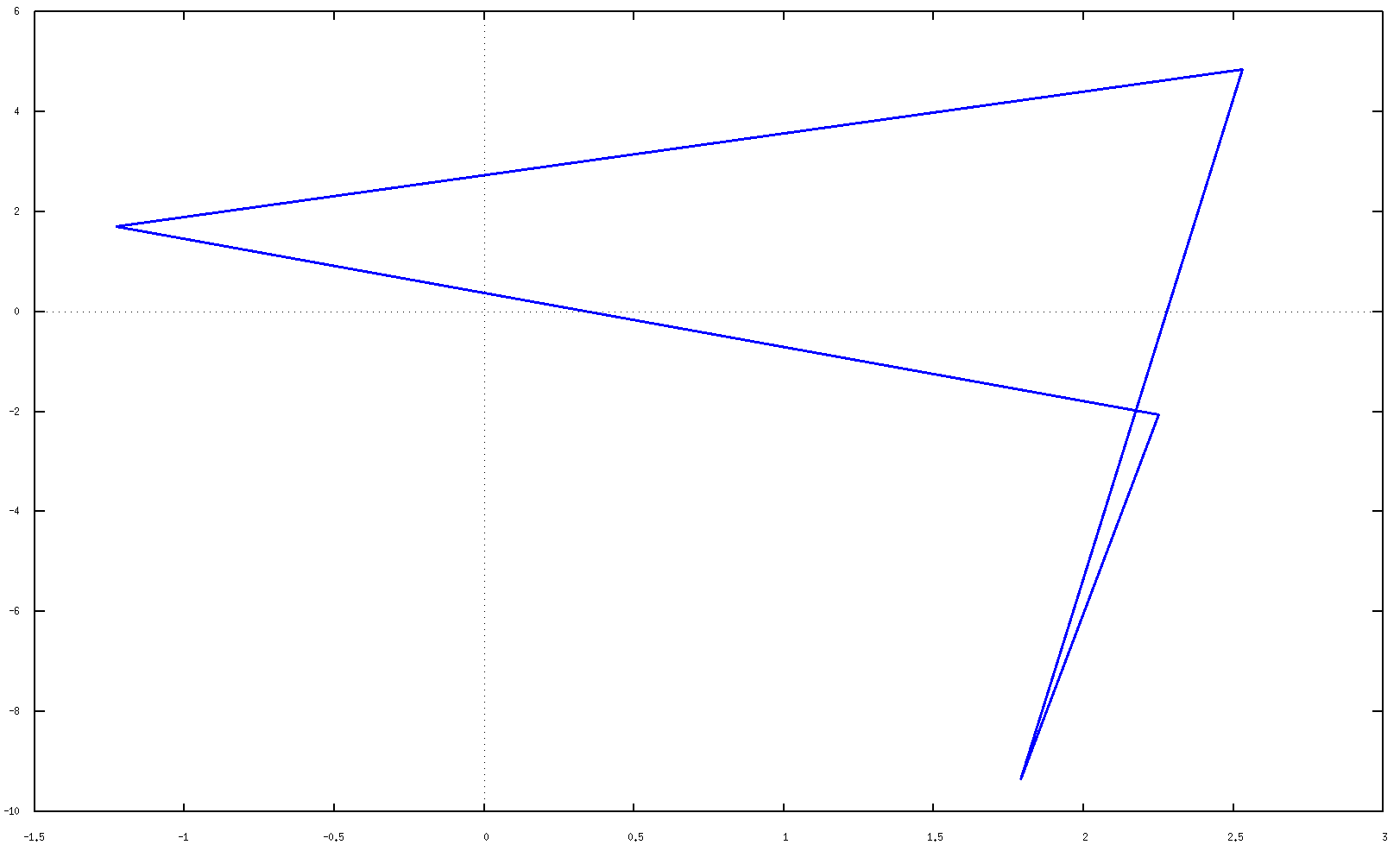}
\includegraphics[width=0.24\textwidth]{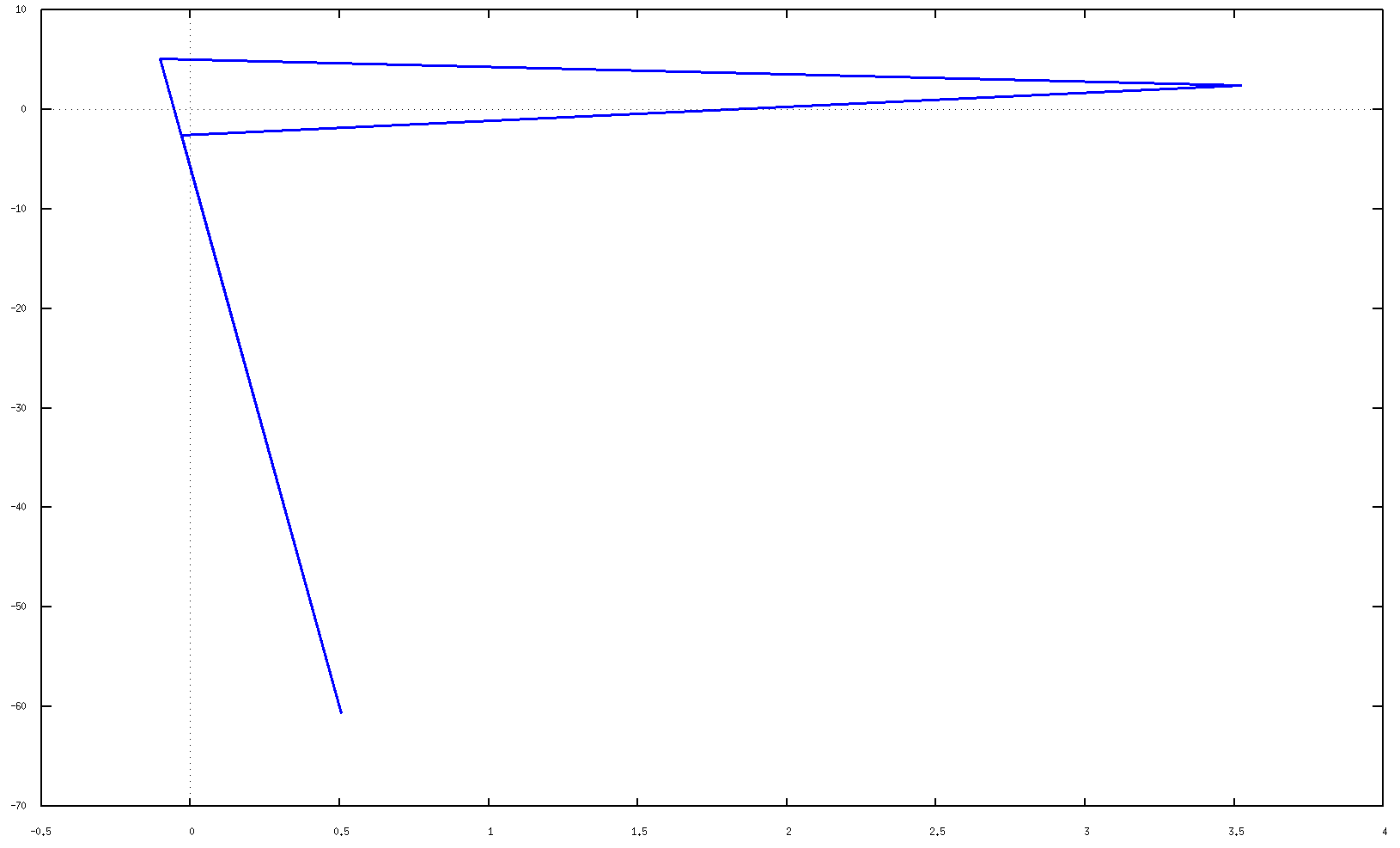}
\includegraphics[width=0.24\textwidth]{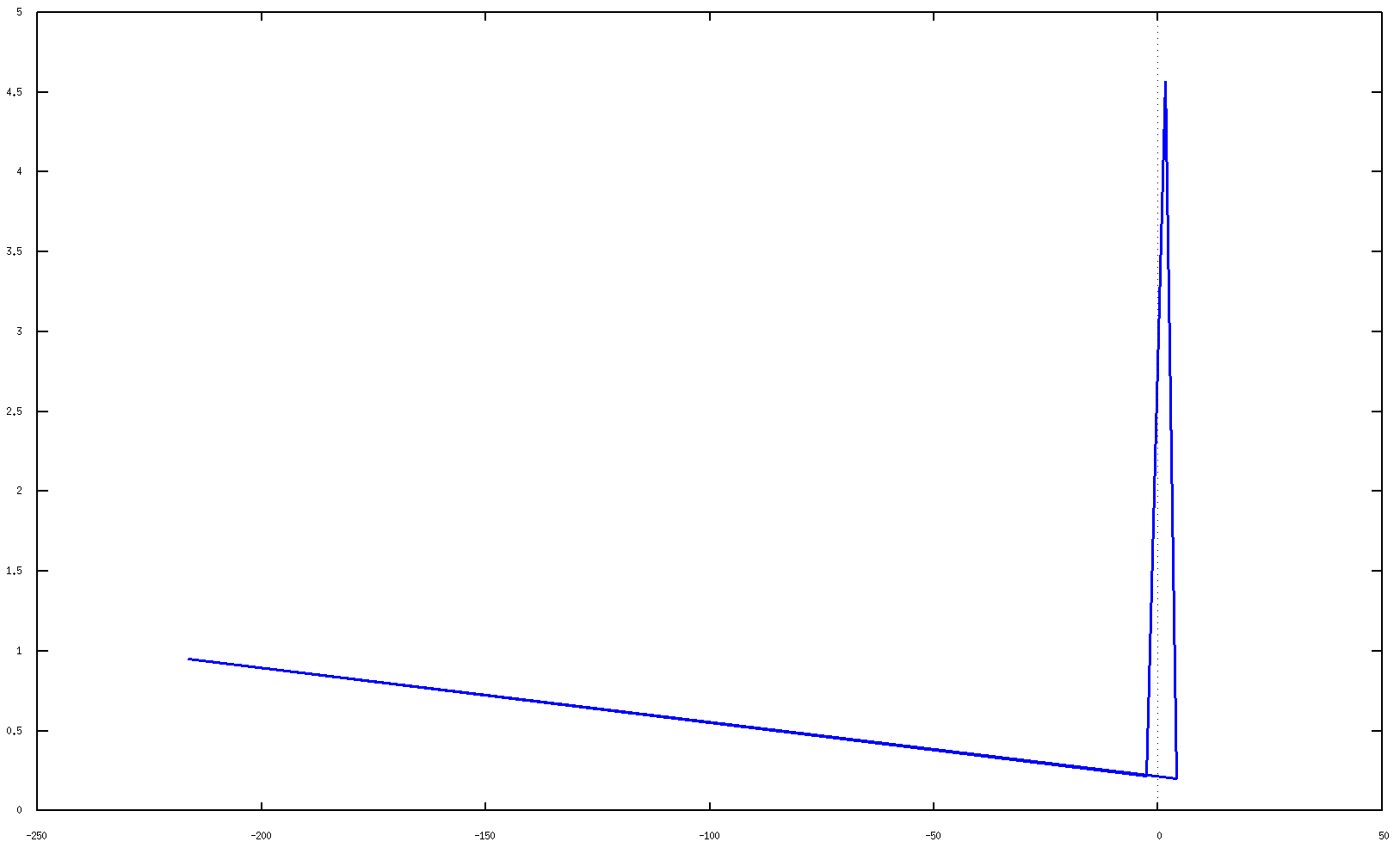}
\caption{Ten real quadrilaterals 
 having the same moments $m_{ij}$ for $i,j \in \lbrace 0,1,2 \rbrace$.}
\label{fig:quads}
\end{figure}
\end{example}

In what follows we consider sets $\mathcal{A}$ with $|\mathcal{A}| = 9$.
Here, the moment variety $\mathcal{M}_{\mathcal{A}}(Q)$ is  typically a hypersurface in~$\mathbb{P}^9$.
The most natural index sets $\mathcal{A}$ arise from partitions
 of the integer $10$. Given a partition
 $\lambda=\{\lambda_0 \geq \lambda_1 \geq \cdots \geq \lambda_s > 0\}$, the
 corresponding index set is
 \begin{align*}
\mathcal{A}_\lambda\,\, := \,\,\lbrace  (0,1), (0,2), \ldots, (0,\lambda_0-1), (1,0),(1,1), \ldots, (1,\lambda_1-1), \ldots, (s,0),\ldots, (s ,\lambda_s-1) \rbrace.
\end{align*}
We simply write $\mathcal{M}_\lambda := \mathcal{M}_{\mathcal{A}_\lambda}$.
For example, $\mathcal{M}_{{4\,3\,2\,1}}(Q)$ is the variety of moments up to order three 
 which was  earlier denoted by $\mathcal{M}_{[3]}(Q)$.
 We determine this hypersurface explicitly.

\begin{theorem} \label{thm:beaquad}
Let $Q$ be a quadrilateral in the plane. 
The moment variety $\mathcal{M}_{[3]}(Q)$ is~a hypersurface in $\PP^9$, whose
defining polynomial has $5100$ terms of degree $({\bf 18},12,12)$. It equals
$$ \begin{small}  \begin{matrix}
 2125764 \, h^6
 \,+\,5484996 \,m_{00}^2 h^4 s
\, -\,1574640 \,m_{00} g h^3
  \, +\,364500 \,m_{00}^3h^3t
     \,+\,3458700 \,m_{00}^4 h^2 s^2     \\
  - 2041200 m_{00}^3 g h s 
  + 472500 m_{00}^5 h s t
  -122500 m_{00}^6 s^3   
  +291600 m_{00}^2 g^2
  -135000 m_{00}^4 g t
  +15625 m_{00}^6 t^2 ,
\end{matrix} \end{small} $$
where the affine invariants in (\ref{ex:affinv33}) are normalized
to have content one and leading monomials
$$ \begin{small} s = m_{00} m_{02} m_{12} m_{30} + \,\cdots \! , \,\,
     t = m_{00}^2 m_{03}^2 m_{30}^2 + \,\cdots \! ,\,\,
     h = m_{00} m_{02} m_{20} + \,\cdots \! ,\,\,
     g = m_{00}^3 m_{02}^3 m_{30}^2 +\, \cdots.
     \end{small}     $$
\end{theorem}

\begin{proof}[Derivation and Proof]
The above formula was found as follows.
By Table~\ref{tab:quadrDegrees} below,   the $\mathbb{Z}^3$-degree of the hypersurface is $(18,12,12)$.
We used Proposition \ref{prop:moliensuccess} to generate affine invariants of this degree with indeterminate coefficients.
By plugging in the moments $m_{ij}$
from various random quadrilaterals, we created a system of
linear  equations in the coefficients. This system was solved which led to the formula above. Independent verification 
of the formula was carried out by checking that it vanishes on the parametrization 
$\,(\CC^2)^8 \rightarrow \mathcal{M}_{[3]}(Q)$. 
\end{proof}

We demonstrate the same technique for another interesting
hypersurface in $\PP^9$ that is also invariant under ${\rm Aff}_2$.
It represents the moments of order $\leq 3$ of probability measures
on the triangle $\Delta_2$ whose densities are linear functions.
This hypersurface is the image of the
$8$-dimensional variety $\PP^2 \times \mathcal{M}_{[4]}(\Delta_2)$ 
under the map into $\PP^9$ whose coordinates are
$$ \qquad M_{ij} \,\, =  \,\, \alpha \cdot m_{i+1,j}\, +\, \beta \cdot  m_{i,j+1}
\, +\, \gamma \cdot m_{i,j}  \quad \qquad {\rm for} \quad
0 \leq i+j \leq 3. $$
Here $(\alpha : \beta : \gamma ) \in \PP^2\,$ and $\,m_{i,j}$ are the moments of
the uniform probability measure on $\Delta_2$.

\begin{proposition} \label{prop:52}
The above hypersurface has degree $({\bf 52},36,36)$. Its defining polynomial~is
$$
\begin{tiny}
\begin{matrix}
 12288754756878336m^{16} s^9
 -125913170530271232 h^2 m^{14} s^8
 -11555266180939776 h m^{15} s^7 t
 -423695444226048 m^{16} s^6 t^2 \\
-242587475329941504 h^4 m^{12} s^7 
-67888179490848768 h^3 m^{13} s^6 t
-2253544388296704 h^2 m^{14} s^5 t^2
+92156256976896 h m^{15} s^4 t^3 \\
+4239929831616 m^{16} s^3 t^4
-2425179321925632 g h m^{13} s^7 
+767341894828032 g m^{14} s^6 t
-1302706722212675584 h^6 m^{10} s^6 \\
-108262506929061888 h^5 m^{11} s^5 t
+673312350928896 h^4 m^{12} s^4 t^2 
+535497484271616 h^3 m^{13} s^3 t^3 
+31959518257152 h^2 m^{14} s^2 t^4 \\
+440798423040 h m^{15} s t^5
+195936798885543936 g h^3 m^{11} s^6
-410140620619776 g h^2 m^{12} s^5 t 
-412398826108747776 g h^6 m^8 s^3 t \\
 -2360537593675776 g h m^{13} s^4 t^2
-89805332054016 g m^{14} s^3 t^3
-486870353365172224 h^8 m^8 s^5
+6819936693387264 h^7 m^9 s^4 t \\
+29422733985054720 h^6 m^{10} s^3 t^2 
+2782917213290496 h^5 m^{11} s^2 t^3 
+58246341746688 h^4 m^{12} s t^4 
-587731230720 h^3 m^{13} t^5  \\
+3602104581095424 g^2 m^{12} s^6 
-157746980481662976 g h^5 m^9 s^5 
-79828890012352512 g h^4 m^{10} s^4 t
-10700934975848448 g h^3 m^{11} s^3 t^2 \\
-668738492301312 g h^2 m^{12} s^2 t^3
-10448555212800 g h m^{13} s t^4
+275499014400 g m^{14} t^5 
+1321196639636946944 h^{10} m^6 s^4 \\
+814698134331457536 h^9 m^7 s^3 t
+92179893357379584 h^8 m^8 s^2 t^2
+2541749079638016 h^7 m^9 s t^3
-13792092880896 h^6 m^{10} t^4 \\
+58678654946770944 g^2 h^2 m^{10} s^5
+16167862146170880 g^2 h m^{11} s^4 t
+705486447968256 g^2 m^{12} s^3 t^2
-1103687847816200192 g h^7 m^7 s^4 \\
+13931406950400 g h^3 m^{11} t^4 
-44584171418419200 g h^5 m^9 s^2 t^2 
-9685512225 m^{16} t^6
-1132386035171328 g h^4 m^{10} s t^3  \\
+7839053087502237696 h^{12} m^4 s^3 
+1352219532013338624 h^{11} m^5 s^2 t 
+51427969540816896 h^{10} m^6 s t^2
-147941222252544 h^9 m^7 t^3 \\
+356552602772570112 g^2 h^4 m^8 s^4
+65355404946702336 g^2 h^3 m^9 s^3 t
+5201278745444352 g^2 h^2 m^{10} s^2 t^2 
+99067782758400 g^2 h m^{11} s t^3 \\
-3265173504000 g^2 m^{12} t^4
-5301992678571900928 g h^9 m^5 s^3
-984505782412247040 g h^8 m^6 s^2 t
-37440870596739072 g h^7 m^7 s t^2 \\
+260713381625856 g h^6 m^8 t^3
+7163309458867617792 h^{14} m^2 s^2
+495888540219998208 h^{13} m^3 s t
-613682107121664 h^{12} m^4 t^2 \\
-33414364526542848 g^3 h m^9 s^4 
-2441030167166976 g^3 m^{10} s^3 t
+1297818789047435264 g^2 h^6 m^6 s^3
+235088951956733952 g^2 h^5 m^7 s^2 t \\
+8250658482290688 g^2 h^4 m^8 s t^2
-132090377011200 g^2 h^3 m^9 t^3 
-7123133303988682752 g h^{11} m^3 s^2
-506754841838616576 g h^{10} m^4 s t \\
+2079004689432576 g h^9 m^5 t^2
+1846757322198614016 h^{16} s
-126388861612851200 g^3 h^3 m^7 s^3 
-17847573389770752 g^3 h^2 m^8 s^2 t \\
-469654673817600 g^3 h m^9 s t^2
+20639121408000 g^3 m^{10} t^3
+2594242435278176256 g^2 h^8 m^4 s^2
+183620365983940608 g^2 h^7 m^5 s t \\
-1848091141472256 g^2 h^6 m^6 t^2
-2445243491429646336 g h^{13} m s
+5610807836540928 g h^{12} m^2 t
+3143555283419136 g^4 m^8 s^3 \\
-408993036765233152 g^3 h^5 m^5 s^2 
-26702361435045888 g^3 h^4 m^6 s t
+626206231756800 g^3 h^3 m^7 t^2
+1246806603479384064 g^2 h^{10} m^2 s \\
-9737274975584256 g^2 h^9 m^3 t 
+22822562857746432 g^4 h^2 m^6 s^2 
+1113255523123200 g^4 h m^7 s t
-73383542784000 g^4 m^8 t^2 \\
-299841218941026304 g^3 h^7 m^3 s 
+5822326385934336 g^3 h^6 m^4 t
-12824703626379264 g^2 h^{12} 
+32389413531025408 g^4 h^4 m^4 s \\
-1484340697497600 g^4 h^3 m^5 t 
+15199648742375424 g^3 h^9 m
-1055531162664960 g^5 h m^5 s
+139156940390400 g^5 m^6 t \\
-6878544743366656 g^4 h^6 m^2
+1407374883553280 g^5 h^3 m^3
-109951162777600 g^6 m^4.
\end{matrix}
\end{tiny}
$$
Here $m= m_{00}$ and $s,t,h,g$ are the
affine invariants in Example \ref{ex:covv33}  and
 Theorem~\ref{thm:beaquad}.
\end{proposition}

We now return to the hypersurfaces $\mathcal{M}_\lambda(Q)$ 
that encode moments
of the uniform probability distribution on a quadrilateral $Q$.
These also live in $\PP^9$ but they
are not invariant under ${\rm Aff}_3$.
We consider arbitrary partitions $\lambda$ of $10$ and notice that 
their total number  is~$42$.

\begin{remark} \rm
\label{rem:quadrDimension}
For every partition $\lambda$ of $10$, except those in the following table, the moment variety $\mathcal{M}_{\lambda}(Q)$ is a hypersurface in $\mathbb{P}^9$.
The dimensions of the remaining moment varieties coming from partitions of $10$ are as follows.
Here $\lambda^c$ denotes the conjugate partition of $\lambda$.
\begin{center}
\begin{tabular}{ccc}
  \Xhline{2\arrayrulewidth}
  $\lambda$ & $\lambda^c$ & $\dim \mathcal{M}_{\lambda}(Q)$ \\
  \hline
 $10$ & $1^{10}$ & 5 \\
 $9\,1$ & $2\,1^8$ & 6 \\
 $8\,2$ & $2^2\,1^6$ & 7 \\
 $8\,1^2$ & $3\,1^7$ & 7
  \\
  \Xhline{2\arrayrulewidth}
\end{tabular}
\end{center}

In light of Theorem~\ref{th:Kathlen}, we find that all equations for moment varieties in this table arise from projections onto a line.
In particular, adding either $m_{10},m_{11},m_{12}$ or $m_{10},m_{11},m_{20}$ or $m_{10},m_{20},m_{30}$ to the moments $m_{00},m_{01},\ldots,m_{06}$ does not impose any new relations.
The hypersurfaces $\mathcal{M}_{{7\,3}}$, $\mathcal{M}_{{7\,2\,1}}$ and $\mathcal{M}_{{7\,1^3}}$ are all cut out by the same Hankel determinant~\eqref{eq:HankelQuadr}.
\end{remark}

We now come to the census of mixed relations we are interested in.
These are the moment hypersurfaces $\mathcal{M}_\lambda(Q)$ in $\PP^9$ that are not featured in
Remark~\ref{rem:quadrDimension}. One of them is defined by the polynomial
of degree $18$ seen in Theorem~\ref{thm:beaquad}. The other hypersurfaces
are not invariant under ${\rm Aff}_3$. We computed all of them
using numerical algebraic geometry. Here is the result:

\begin{theorem} \label{thm:42partitions}
Table~\ref{tab:quadrDegrees} lists the $\mathbb{Z}^3$-degrees of the moment hypersurfaces $\mathcal{M}_\lambda(Q)$ in $\mathbb{P}^9$, where
$Q$ is a quadrilateral and $\lambda$ is a partition of $10$. 
We also report  the size of the general fiber of the map
$\varphi_\lambda:(\mathbb{C}^2)^4 \to \mathcal{M}_\lambda(Q)$
which sends the vertices of $Q$ to the moments indexed by~$\lambda$. 

\begin{table}[h]
\centering
\begin{tabular}{cccc}
  \Xhline{2\arrayrulewidth}
  $\lambda$ & $\lambda^c$ & $\deg \mathcal{M}_{\lambda}(Q)$ & $\deg \varphi_\lambda$ \\
  \hline
$7\,3$ & $2^3\,1^4$ & $(5,10,0)$ & $144$ \\
$7\,2\,1$ & $3\,2\,1^5$ & $(5,10,0)$ & $144$ \\
$7\,1^2$ & $4\,1^6$ & $(5,10,0)$ & $144$ \\
$6\,4$ & $2^4\,1^2$ & $(27,3,36)$ & $8$ \\
$6\,3\,1$ & $3\,2^2\,1^3$ & $(51,6,54)$ & $8$ \\
$6\,2^2$ & $3^2\,1^4$ & $(96,12,90)$ & $8$ \\
$6\,2\,1^2$ & $4\,2\,1^4$ & $(136,18,126)$ & $8$ \\
$6\,1^4$ & $5\,1^5$ & $(480,72,424)$ & $8$ \\
$5^2$ & $2^5$ & $(33,6,39)$ & $8$ \\
$5\,4\,1$ & $3\,2^3\,1$ & $(36,6,36)$ & $8$ \\
$5\,3\,2$ & $3^2\,2\,1^2$ & $(42,12,36)$ & $8$ \\
$5\,3\,1^2$ & $4\,2^2\,1^2$ & $(60,18,48)$ & $8$ \\
$5\,2^2\,1$ & $4\,3\,1^3$ & $(72,36,42)$ & $8$ \\
$5\,2\,1^3$ & $5 \,2\,1^3$ & $(139,70,72)$ & $8$ \\
$4^2\,2$ & $3^2\,2^2$ & $(42,16,32)$ & $8$ \\
$4^2\,1^2$ & $4\,2^3$ & $(60,24,42)$ & $8$ \\
$4\,3^2$ & $3^3\,1$ & $(47,20,34)$ & $8$ \\
$4\,3\,2\,1$ & $4\,3\,2\,1$ & $(18,12,12)$ & $8$ \\
  \Xhline{2\arrayrulewidth}
\end{tabular}
\caption{Degrees of moment hypersurfaces of quadrilaterals.}
\label{tab:quadrDegrees}
\end{table}
\end{theorem}

\begin{proof}[Derivation and Proof]
This is based on numerical computations. We started out with
\texttt{Bertini} \cite{Bertini}, but then we mainly used the 
 \texttt{Julia} package \texttt{HomotopyContinuation.jl}~\cite{homotopyCont}.
 
Consider the parametrization of the affine cone over the moment hypersurface $\mathcal{M}_\lambda(Q)$ given by $\mathbb{C}^9 \to \mathbb{C}^{10}, (t,X) \mapsto t \cdot \varphi_\lambda(X)$.
Let us first describe how we compute the usual degree of this affine cone in $\CC^{10}$.
We pick a random point on the  cone together with a random line passing through this point. Our goal is to compute all intersection points of the line with the cone. 
We do this via numerical monodromy, i.e. we move the line around and track the already known intersection point. 
When the original line is reached again, we might have found a new solution.
These monodromy loops are executed until no new solutions are found. 
To verify that all solutions have been found, we applied the trace test \cite[\S 10.2.1]{Bertini}.

To compute the other two coordinates
in the $\mathbb{Z}^3$-degree of the moment hypersurface $\mathcal{M}_\lambda(Q)$, we proceed as above, but the line is now replaced by a monomial curve. For the middle coordinate of
the $\ZZ^3$-degree, we use the curve in $\CC^{10}$ with parametric representation
$$
s \mapsto \left( p_1 + s^{i_1} v_1, \; p_2 + s^{i_2} v_2, \;\ldots, \; p_{10} + s^{i_{10}} v_{10} \right).
$$
Here $p$ and $v$ are random vectors in $ \mathbb{C}^{10}$.
 The moments indexed by the partition $\lambda$ appear in the 
 order $m_{i_1,j_1}, m_{i_2,j_2}, \ldots, m_{i_{10},j_{10}}$.
Analogously, for the last entry in the $\mathbb{Z}^3$-degree, we use the
monomial curve in $\CC^{10}$ parametrized by 
$s \mapsto \left( p_1 + s^{j_1} v_1,  \; p_2 + s^{j_2} v_2,  \ldots, p_{10} + s^{j_{10}} v_{10} \right)$.

In each case, we solve a square system of $10$ polynomial
equations in $10$ unknowns $s,t,x_{11},\ldots,x_{42}$.
The number of solutions is the desired degree in $\CC^{10}$
times the degree of the map $\varphi_{\lambda}$. For instance,
the number of solutions $(s,t,x_{11},\ldots,x_{42})$
 for $\lambda = (4,3,2,1)$ equals $144$.
The solutions  form $18$ clusters of size $8$, where each cluster
consists of all solutions   that map to the
same point on the affine cone. This is how the degree $18$
was first determined. It allowed us to make the ansatz 
that eventually led to the invariant in Theorem~\ref{thm:beaquad}.
\end{proof}

The use of invariant theory of the affine group ${\rm Aff}_d$
was essential for computing the moment hypersurfaces 
in Theorem~\ref{thm:beaquad} and Proposition~\ref{prop:52}.
However this method does not directly apply to moment varieties of codimension two or more.
For such moment varieties, the minimal generators of the ideal
form an invariant vector space, but the individual generators
are not invariants. In such a situation, one might employ
 representation theory of ${\rm Aff}_d$. We shall demonstrate this for the 
 moment variety $\mathcal{M}_{[3]}(\Delta_3)$ in Conjecture \ref{conj:cummom}.

\begin{proposition} \label{prop:tetra}
The $\mathrm{Aff}_3$-module $V$ spanned by the $90$ quintics 
that vanish on $\mathcal{M}_{[3]}(\Delta_3) $ in $ \PP^{19}$
is the direct sum of two indecomposable $\mathrm{Aff}_3$-modules
 $V_1$ and $V_2$, each of dimension~$45$.
As a $\mathrm{GL}_3$-module, $V$ decomposes into $12$ irreducibles:
$V_1$ and $V_2$ split into six irreducible $\mathrm{GL}_3$-modules each. 
Table~\ref{tab:GL3modules} lists
the highest weights of these $\mathrm{GL}_3$-modules and their dimensions.

\begin{table}[h]
\centering
\begin{tabular}{?c|cccccc?}
  $V_1$ & $(3,3,4)$ & $ (3,4,4)$ & $ (2,4,4)$ & $ (2,3,4)$ & $ (1,4,4) $ & $ (1,3,4)$ \\
 $V_2 $ & $ (2,2,3)$ & $ (2,3,3)$ & $ (2,2,4)$ & $ (2,3,4)$ & $ (2,2,5)$ & $ (2,3,5)$ \\
 $\dim$ & $3$ & $3$ & $6$ & $8$ & $10$ & $15$
\end{tabular}
\caption{Decomposition of the $\mathrm{Aff}_3$-modules $V_1$ and $V_2$ into irreducible $\mathrm{GL}_3$-modules.}
\label{tab:GL3modules}
\end{table}
\end{proposition}

\begin{proof}
The \emph{weight} of a polynomial is given by its $\mathbb{Z}^{4}$-grading.
Each \emph{isotypical component} of $V$ as a $\mathrm{GL}_3$-module is spanned by all polynomials in $V$ having the same fixed $\mathbb{Z}^4$-degree.
This isotypical decomposition consists of $43$ vector spaces with dimensions $1$, $2$, $4$ or $6$.

For each isotypical component, we computed its $U_3$-invariant polynomials, where $U_3 \subset \mathrm{GL}_3$ is the subgroup of upper triangular matrices with diagonal $(1,1,1)$.
Ten isotypical components contain exactly one $U_3$-invariant (up to scaling),
while the component with weight $(2,3,4)$ has a two-dimensional subspace of $U_3$-invariant polynomials; see Table~\ref{tab:GL3modules}.
Each  $U_3$-invariant  generates an irreducible $\mathrm{GL}_3$-module.
Ten of these irreducible modules in $V$ are unique. 
The two irreducible $\mathrm{GL}_3$-modules with highest weight $(2,3,4)$ are not unique.

Finally, we studied which of the described irreducible $\mathrm{GL}_3$-modules merge when we add translation, i.e. when we act on $V$ by the whole affine group $\mathrm{Aff}_3$.
The ten unique $\mathrm{GL}_3$-modules get merged into two clusters, as seen in Table~\ref{tab:GL3modules}.
Moreover, there is a unique way of choosing two $\mathrm{GL}_3$-modules with highest weight $(2,3,4)$ such that acting with the affine group on one of these modules stays within one of the two clusters in Table~\ref{tab:GL3modules}.
\end{proof}

\section{Outlook}
\label{sec7}

Moment varieties furnish an algebro-geometric representation for 
various probability measures on $\RR^d$. In this
article we focused on measures that are associated with convex polytopes.
We were able to determine their moment varieties for a range
of interesting cases. However, this is just the beginning.
Many questions remain open, and we see considerable
potential for further developing our algebraic tools,
so that they become  practical for inverse problems.

This section discusses a number of open problems and directions
for future research. It also offers a perspective on some
aspects of moment varieties not discussed in Sections \ref{sec2}--\ref{sec6}.

\medskip
\noindent 
{\bf Adjoints and Wachspress Varieties.} 
%
At the end of Section \ref{sec2} we defined the adjoint moment 
variety $\,\mathcal{M}_{\rm Ad}(\cP)\,$ for a given combinatorial type $\cP$, 
but we did not state any results on this topic.
The variety $\mathcal{M}_{\rm Ad}(\cP)$ is the moduli space
for the Wachspress varieties of the polytopes  in the class $\cP$.
The study of Wachspress varieties and their moduli
is a promising direction at the interface of geometric
combinatorics  and algebraic geometry (see~\cite{KR}). It extends the familiar repertoire
of toric varieties.

A concrete open problem is to compute the
adjoint moment variety $\mathcal{M}_{\rm Ad}(\cP)$
in the smallest cases where 
the ambient dimension $\binom{n-1}{d}-1$
exceeds the number $nd$ of parameters.
This happens for polytopes 
with $n=8$ vertices
in dimensions $d=2,3,4$.
Another interesting case is $d=2$ and $n=7$.
Here the adjoint is a plane curve of degree $4$,
so it has $14$ parameters. It is parametrized
by the $14$ vertex coordinates of a heptagon.
What is the degree of this map?  It would be
worthwhile to study  the geometry of this map,
in light of the beautiful classical
connections  \cite[\S 6.3.3]{Dolgachev}
between genus $3$ curves and
del~Pezzo surfaces of degree $2$.

\medskip\noindent 
{\bf Step Functions.}
  It can be shown that mixtures
 of uniform distributions of line segments are
 algebraically identifiable whenever this is permitted
 by the parameter count.
 To be precise, the delicate algebro-geometric proof for
  mixtures of univariate Gaussians that is
 given in \cite[Section 2]{ARS} can be transferred to mixtures of
 line segments. The point of departure for this transfer argument is the proof of
\cite[Lemma 4]{ARS} which holds verbatim for the matrix in (\ref{eq:3rows}).

This opens the door to moment varieties of distributions whose
density is a  step function on the line $\RR^1$.
Indeed, each step function is a mixture of uniform distributions on
line segments. Since mixture models correspond to secant varieties in $\PP^r$,
we can phrase our question as follows: study the secant varieties of the
surfaces $\mathcal{M}_{\{\!\{ r \}\!\}}(1,2)$ in Example \ref{ex:einszwei}.
Pearson's hypersurface of degree $39$ in \cite[Theorem 1]{AFS}
suggests that this will not be easy.

\medskip\noindent 
{\bf Recovery Algorithms.}  Theorem~\ref{th:Kathlen} characterizes all relations among axial moments  of a polytope $P$ for any fixed axis. From this one can recover
  the projections of all vertices of $P$ onto that axis.
  Different variations of this result are known in the literature; see e.g.~\cite{GLPR}. 
  On the other hand, in order to uniquely recover  a polytope $P$ in $ \RR^d$  using 
  axial  moments, one has  to know the projections of its vertices on at least $d+1$ different lines in $\RR^d$. The moments on $d+1$ lines are highly dependent.
  For instance, for $d=2$ and $P$ a quadrilateral, the  $\lambda = 6 \, 1^4$ entry 
  in Table \ref{tab:quadrDegrees} reveals a relation of degree $480$ among moments on two axes.
Understanding such dependencies among the axial moments for general polytopes
seems difficult, but it is an important step towards developing more
advanced recovery algorithms.
This issue is related to multidimensional variants of Prony's method.
Indeed, the Hankel matrix \eqref{eq:hankelmatrix} which connects
polytopal densities and its node points on $\RR^1$ 
with the axial moments is analogous to that for  the classical Prony system \cite{GSY}.
Extending known results about the Prony system to our setting
in $\RR^d$ may lead to applications in signal processing.

\medskip\noindent 
{\bf Multisymmetric Functions.}
Let $\RR[X]$ denote the ring of polynomials in the entries
of an $n \times d$ matrix of unknowns $X = (x_{kl})$. The symmetric group $S_n$
acts on $\RR[X]$ by permuting the rows of $X$.
Following Dalbec \cite{Dal}, we write $\Lambda_{d,n} = \RR[X]^{S_n}$
for the ring of invariants under this action. In words, $\Lambda_{d,n}$ is the
ring of multisymmetric functions for $n$ vectors in $d$-space.

The case $n=d+1$ appeared in Section~\ref{sec4}. 
 Proposition~\ref{prop:explicitmI} and Corollary~\ref{cor:cumu} imply
 that the moments of simplices in $\RR^d$ generate the ring $\Lambda_{d,d+1}$.
 Indeed, the moments and the cumulants generate the same algebra, and
 the  cumulants coincide with the power sum multisymmetric polynomials.
By \cite[Theorem 1.2]{Dal}, the latter are known to generate $\Lambda_{d,n}$ for any $n$.
Furthermore, our Proposition \ref{prop:eachcumu} is closely related to the
well-known fact  (cf.~\cite[Theorem 1.3]{Dal})
that elementary multisymmetric polynomials also generate the algebra $\Lambda_{d,n}$.
  
The discussion at the end of Section \ref{sec4} shows that, for any $n > d$,
the ring $\Lambda_{d,n}  $ arises from our polytopal measures.
Namely, consider the projection of an $(n-1)$-simplex to a subspace $\RR^d$.
Suppose that  the image is a $d$-polytope with $n$ vertices.
The moments of the induced polytopal measure are multisymmetric polynomials in
$\Lambda_{d,n}$, and, in fact, these moments generate the invariant ring $\Lambda_{d,n}$.
Therefore we obtain all possible rings of multisymmetric polynomials as  special cases of the rings of moments of simplices and their projections.  It is known in algebraic combinatorics
that these rings are quite complicated, see e.g. \cite{Hai}. 

\medskip\noindent 
{\bf Symmetry and Invariants.}  We demonstrated in Section~\ref{sec5} that
invariants of the affine group can be determined from covariants of the
general linear group, and this was used in Section~\ref{sec6} to give explicit
formulas for two specific moment hypersurfaces in $\PP^9$. In the case
of moment varieties of codimension $\geq 2$, we do not really know 
how to take advantage of symmetries arising from the affine group ${\rm Aff}_d$.
It would be desirable to understand this.

\medskip\noindent 
{\bf More Hypersurfaces.} 
In Theorem \ref{thm:42partitions} we determined many
moment hypersurfaces of quadrilaterals in $\PP^9$, one for each
partition $\lambda$ of the integer $10$. Our computations were based
on methods from numerical algebraic geometry. One could
try to push this further, either to pentagons ($d=2,n=5$)
or to tetrahedra $(d=3,n=4)$. In the former case
we would aim for moment hypersurfaces in $\PP^{11}$ associated
with partitions of $12$, and in the
latter case we would seek  moment hypersurfaces in $\PP^{13}$
associated with plane partitions of $14$. The remark after
Proposition  \ref{prop:moliensuccess} suggests the following problem
for numerical algebraic geometry:
compute the degrees of the moment hypersurfaces
$\,\mathcal{M}_{[6]}(\text{$13$-gon}) \subset \PP^{27}\,$ 
and $\,\mathcal{M}_{[3]}({\rm octahedron} )\subset \PP^{19}$.
 
\medskip\noindent 
{\bf Special Subvarieties.}
It would be interesting to study the singular loci of moment varieties
as well as the subvarieties whose points correspond to degenerate geometric configurations. 
This was discussed for the cubic surface in Figure \ref{fig:M13} but we never returned to that topic.
 
\medskip\noindent 
{\bf Moment Rings of Polytopes.}  Fix a combinatorial type $\mathcal{P}$ of
simplicial polytopes in $\RR^d$. We define  the moment ring $\fM_\cP$
to be the subalgebra of the rational function field $\RR(X)$ that is generated
by the moments $m_I(X)$ for $\mathcal{P}$ where $I$ runs over  $\NN^d$.
We can realize $\fM_\cP$ as the subalgebra of the polynomial ring $\RR[X]$,
 generated by the numerators $m_I(X) \cdot {\rm vol}(X)$.
These products are polynomials in the $nd$ unknowns $x_{kl}$ 
by Theorem~\ref{thm:nmgf_polytope}. 
If $P$ is the $d$-simplex  then the moment ring $\fM_\cP$
is the ring $\Lambda_{d,d+1}$ of multisymmetric polynomials, as discussed above.
A priori, it is not even clear that $\fM_\cP$ is a Noetherian ring. However,
we strongly believe this. In other words, we conjecture
that $\fM_\cP$ is finitely generated. It would be very interesting to
identify explicit generators, or, at least, to find
degree bounds for the generators of $\fM_\cP$.
The same question makes sense for the moment rings
 that are analogous to $\Lambda_{d,n}$ for $n > d+1$.
To be specific, we seek the subalgebra of $\RR(X)$ that is
generated by  all moments  of univariate polytopal measures of type $(d,n)$. 
A natural place to start is the case of the convex $n$-gon in the plane.
Here we might take advantage of the dihedral group acting
on the $n$ vertices.
In the case of the ring generated by all harmonic moments of plane polygons such study was carried out in \cite{BuFrSh}.

The group of symmetries of the combinatorial type $\cP$ acts on its moment ring 
$\fM_\cP$. This explains why the moment ring of a simplex consists of multisymmetric functions and why the dihedral group acts on the moment rings of $n$-gons. Of course,
there are many other types of simplicial polytopes with interesting symmetry groups.
  How about the octahedron?

\bigskip \bigskip \bigskip \medskip

\paragraph{Acknowledgments.}
We thank Jan Draisma, Frank Grosshans and Hanspeter Kraft  for 
communications on invariant theory.
We are grateful to Taylor Brysiewicz, Paul Breiding and Sascha Timme for
helping us with our experiments using numerical algebraic geometry.
Kathl\'en Kohn and Boris Shapiro are grateful to the MPI MIS in Leipzig 
for the hospitality in June 2018 where this project was initiated.
Bernd Sturmfels also acknowledges partial support
from the Einstein Foundation Berlin and the US National Science Foundation.

\bigskip
\bigskip

\begin{small}

\end{small}

\bigskip \bigskip
\bigskip

\noindent
\footnotesize {\bf Authors' addresses:}

\smallskip 

\noindent Kathl\'en Kohn,
ICERM, Brown University
\hfill {\tt kathlen.korn@gmail.com}

\noindent
Boris Shapiro, Stockholm University
\hfill {\tt shapiro@math.su.se}

\noindent Bernd Sturmfels,
 \  MPI-MiS Leipzig and
UC  Berkeley \hfill  {\tt bernd@mis.mpg.de}

\end{document}